\documentclass{amsart}

\usepackage{color}
\usepackage{amsthm,amssymb,verbatim}
\usepackage{graphicx}
\usepackage{enumerate}
\usepackage{hyperref}
\usepackage{caption}
\usepackage{subcaption}

\newtheorem{claim}{Claim}

    \newtheorem{theorem}    {Theorem}      
    \newtheorem{lemma}      [theorem]       {Lemma}
    \newtheorem{corollary}  [theorem]     {Corollary}
    \newtheorem{proposition}       [theorem]       {Proposition}
    
    \newtheorem*{claim*}{Claim}
    \newtheorem*{theorem*}{Theorem}
    \newtheorem*{corollary*}{Corollary}

    \theoremstyle{definition}

    \theoremstyle{definition}
    \newtheorem*{definition*}{Definition}
    
    \newtheorem*{remark*}{Remark}
    \newtheorem*{remarks*}{Remarks}

\newcommand{\ddx}{\frac{\partial}{\partial x}}
\newcommand{\ddy}{\frac{\partial}{\partial y}}
\newcommand{\ddz}{\frac{\partial}{\partial z}}
\newcommand{\ddzbar}{\frac{\partial}{\partial \overline{z}}}

\newcommand{\eps}{\epsilon}

\newcommand{\nn}{\mathbf{n}}
\newcommand{\vv}{\mathbf{v}}
\newcommand{\uu}{\mathbf{u}}

\newcommand{\Ss}{\mathbf{S}}
\newcommand{\BB}{\mathbf{B}}
\newcommand{\RR}{\mathbf{R}}
\newcommand{\CC}{\mathbf{C}}
\newcommand{\ee}{\mathbf{e}}
\newcommand{\Tan}{\operatorname{Tan}}
\newcommand{\area}{\operatorname{area}}
\newcommand{\dist}{\operatorname{dist}}

\newcommand{\Cc}{\mathcal{C}}

\newcommand{\ddt}{\left(\frac{d}{dt}\right)}
\newcommand{\ddr}{\left(\frac{d}{dr}\right)}

\newcommand{\wform}{\eta}  

\newcommand{\Div}{\operatorname{div}}

\newcommand{\tM}{\tilde M}

\newcommand{\lecture}[1]{\clearpage \specialsection{\bf #1}}

\usepackage{hyperref}
\usepackage[alphabetic, msc-links]{amsrefs}

\begin{document}

\title{Lectures on Minimal Surface Theory}
\author{Brian White}
\address{Department of Mathematics\\ Stanford University\\ Stanford, CA 94305}
\thanks{The author was partially funded by NSF grants~DMS--1105330 and DMS--1404282}
\email{bcwhite@stanford.edu}
\date{August 1, 2013. Revised January 16, 2016.}

\maketitle


\tableofcontents












\clearpage

\specialsection*{Introduction}
 
 Minimal surfaces have been studied by mathematicians
 for several centuries, and they continue to be a very active
 area of research.   Minimal surfaces have also
 proved to be a useful tool in other areas, such as in 
 the Willmore problem and in general relativity.  
 Furthermore, various general techniques in geometric analysis were
 first discovered in the study of minimal surfaces.

 The following notes are slightly expanded versions 
 of four lectures presented at the 2013 summer program 
 of the Institute for Advanced Study and the Park City Mathematics Institute.
 The goal was to give beginning graduate students
 an introduction to some of the most important basic
 facts and ideas in minimal surface theory.
 I have kept prerequisites to a minimum: 
 the reader should know basic complex analysis
 and elementary differential geometry (in particular, the second fundamental form
 and the Gauss-Bonnet Theorem).  
 
 For readers who wish to pursue the subject further, there
 are a number of good books available, such as \cite{colding-minicozzi},
 \cite{lawson}, and \cite{dierkes-et-al}.
 Readers may also enjoy exploring Matthias Weber's extensive online
 collection of minimal surface images:
 \url{http://www.indiana.edu/~minimal/archive/}.  Except for Figure~\ref{genus-two-figure}, all
 of the illustrations in these notes are taken from that collection.
 
 If I had a little more time, I would have talked more about
 the maximum principle and about the structure of the intersection
 set for pairs of minimal surfaces.
 See for example \cite{colding-minicozzi}*{1.\S7, 6.\S1, 6.\S2}.

If I had a lot more time, I would have talked about geometric
measure theory, which has had and continues to have an enormous impact on minimal
surface theory.   Almgren's 1969 expository article~\cite{almgren-expo} remains an excellent introduction.
Morgan's book~\cite{morgan} is a very readable account of the main concepts and results.
In many cases, he describes the key ideas of proofs without giving any details.
For complete proofs, I recommend~\cite{simon-book}.

In the last few years, there have been a 
number of spectacular breakthroughs in minimal surface theory
that are not mentioned here.  See~\cite{meeks-perez-survey}
for a survey of many of the recent results.

I would like to thank Alessandro Carlotto for running problem sessions
for the graduate students attending my lectures and for  carefully
reading an early version of these notes and making a great many suggestions.
I would also like to thank David Hoffman for additional suggestions.
The notes are much improved as a result of their input.



%
 \lecture{The First Variation Formula and Consequences}\label{lecture1}
 
 Let $M$ be an $m$-dimensional surface in $\RR^n$ with boundary $\Gamma$.
We say that $M$ is a {\bf least-area surface} if its area is less than or equal to the area of any
other surface having the same boundary.  To make this definition precise, one has to specify the 
exact meaning of the words ``surface", ``area", and ``boundary".
For example, do we require the surfaces to be oriented?
But for the moment we will be vague about such matters, since 
the topics we consider now are independent of such specifications.

In the Plateau problem, one asks:
given a boundary $\Gamma$ in Euclidean space (or, more generally, in 
a Riemannian manifold), does there exist a least-area
surface $M$ with boundary $\Gamma$?  If so, how smooth is $M$?
For example, if $m=1$ and $\Gamma$ consists of a pair of points in $\RR^n$,
then the solution $M$
of the Plateau problem is the straight line segment joining them.  

In general, however, even proving existence is very nontrivial.  Indeed, in 1936 Jesse Douglas won
 one of the first two\footnote{Ahlfors won the other one.}
Fields Medals
 for his existence and regularity theorems for the $m=2$ case 
of the Plateau problem.  We will discuss those results in lecture~\ref{lecture4}.

Now we consider a related question: given a surface $M$, how do we tell
if it a least-area surface?
In general, it is very hard to tell, but the ``first-derivative test" provides
a necessary condition for $M$ to be a least-area surface:
If $M_t$ is a one-parameter family of surfaces each with boundary $\Gamma$, 
and if $M_0=M$,
 then
\[
    \ddt_{t=0}\area(M_t)
\]
should be $0$.

For the test to be useful, we need a way of calculating the first derivative:

\begin{theorem}[The first variational formula]\label{first-variation-theorem}
Let $M$ be a compact $m$-dimensional manifold in $\RR^n$.
Let $\phi_t:M\to \RR^n$ be a smooth one-parameter family of smooth maps
such that
\[
    \phi_0(p)\equiv p.
\]
Let $X(p)= \ddt_{t=0} \phi_t(p)$ be the initial velocity vectorfield.
Then
\begin{align*}
\ddt_{t=0}\area(\phi_tM)
&=  \int_M \Div_M(X)\,dS  \\
&=  \int_{\partial M} X\cdot \nu_{\partial M}\,ds  -  \int_M H\cdot X\,dS,
\end{align*}
where
\[
\Div_MX := \sum_{i=1}^m \ee_i \cdot \nabla_{\ee_i}X
\]
for any orthonormal basis
$
\ee_1(p), \dots, \ee_m(p) 
$
of $\Tan_pM$, where $H(p)$ is mean curvature vector of $M$ at $p$,
and where $\nu_{\partial M}(x)$ is the unit vector in the tangent plane to $M$ at $x$
that is normal to $\partial M$ and 
that points away from $M$.\footnote{The notation $\nu_{\partial M}$ suggests that the
vector depends only on $\partial M$.  However, it depends on $M$ and thus
arguably should be written in some other way such as $\nu_{\partial M; M}$.}
\end{theorem}

It is perhaps better to refer to the first equality
\[
\ddt_{t=0}\area(\phi_tM)
=  \int_M \Div_M(X)\,dS,
\]
as the ``first variation formula" and to the second
equality
\[
  \int_M \Div_MX \, dS = \int_{\partial M}X\cdot \nu_{\partial M}\,ds -  \int_M H\cdot X\,dS
\]
as the ``generalized divergence theorem".  
Note that when the vectorfield $X$ is tangent
to $M$, the generalized divergence theorem is just the ordinary divergence theorem.

\begin{proof}[Proof of the first variation formula]
Note that
\[
    \area(\phi_tM) =  \int_M J_m(D\phi_t) \,dS
\]
where $J_m(D\phi_t)$ is the Jacobian determinant $\sqrt{ \det (D\phi^T D\phi)}$,
so
\[
   \ddt_{t=0} \area(\phi_tM)  = \int_M \ddt_{t=0} J_m(D\phi_t)\,dS.
\]
Thus it suffices to calculate that
\[
   \ddt_{t=0} J_m(D\phi_t) = \Div_MX.
\]
By definition of $X$,
\[
 \phi_t(p) = p + t\, X(p)  + o(t),
\]
so
\[
D\phi_t(\ee_i)= \nabla_{\ee_i}(\phi_t) \cong \nabla_{\ee_i}(p + t X(p)) = \ee_i + t\nabla_{\ee_i}X,
\]
where $a \cong b$ means $a- b =o(t)$.  Thus
\begin{align*}
J_m(D\phi_t) 
&= \sqrt{  \det (D\phi_t(\ee_i) \cdot D\phi_t(\ee_j) )}
\\
&\cong
\sqrt{ \det ( (\ee_i + t \nabla_{\ee_i}X) \cdot (\ee_j + t\nabla_{\ee_j}X)) }
\\
&\cong
 \sqrt{ \det( \delta_{ij} + t (\ee_i\cdot \nabla_{\ee_j}X + \ee_j\cdot \nabla_{\ee_i}X)) }
\end{align*}
Recall that for a square matrix $A$, 
\[
  \det(I + t A) = 1 + t \operatorname{trace}(A) + o(t).
\]
Thus
\begin{align*}
J_m(D\phi_t)
&\cong
\sqrt{ 1 + 2t \sum_i(\ee_i\cdot \nabla_{\ee_i}X) }
\\
&\cong
\sqrt{ 1 + 2t \Div_MX}  \cong  1 + t \Div_MX.  
\end{align*}
\end{proof}

\begin{proof}[Proof of the generalized divergence theorem]
Splitting $X$ into the portion $X^T$ tangent to $M$ and the portion $X^N$ normal
to $M$, we have, by the ordinary divergence theorem,
\begin{align*}
\int_M \Div_MX\,dS 
&=
\int_M\Div_M(X^T)\,dS  + \int_M \Div_M(X^N)\,dS
\\
&= \int_{\partial M} X^T \cdot \nu_{\partial M}\,ds + \int_M \Div_M(X^N)\,dS
\\
&= \int_{\partial M} X\cdot \nu_{\partial M}\,ds + \int_M\Div_M(X^N)\,dS.
\end{align*}
Thus we need only show that $\Div_M(X^N)= -H\cdot X$.
Using the summation convention,
\begin{align*}
 \Div_M (X^N) 
 &= \ee_i \cdot \nabla_{\ee_i}(X^N)
 \\
 &= \nabla_{\ee_i}(\ee_i\cdot X^N) - (\nabla_{\ee_i}\ee_i)\cdot X^N
 \\
 &= \nabla_{\ee_i}(0)   - (\nabla_{\ee_i}\ee_i)^N \cdot X
 \\
 &=
 -H\cdot X.  \qquad 
\end{align*} 
\end{proof}

\begin{remark*}
The first variation formula and the generalized divergence theorem
are also true for submanifolds of a Riemannian manifold $N$.
They can be deduced from the special case by 
isometrically embedding\footnote{Such an isometric embedding exists by the Nash
Embedding Theorem~\cite{nash-embedding}.  
A more elementary proof was discovered by G\"unther~\cite{gunther}.   
See~\cite{yang-gunther} for a very readable account of G\"unther's proof.}
$N$ in a Euclidean space $\RR^n$.
The first variation formula in $N$ then becomes
a special case of the first variation formula in $\RR^n$.
To see that the generalized divergence theorem is true in $N$, note
that the submanifold $M\subset N\subset \RR^n$ has two mean
curvatures vectorfields: the mean curvature $H_N$ of $M$ as a submanifold of $N$ and the
mean curvature $H$ of $M$ as a submanifold of $\RR^n$.   
Note that $H_N(p)$ is the orthogonal projection of $H$ to $\Tan_pN$. 
If $X$ is a vectorfield on $M$ that is tangent to $N$ (as it will be if $\phi_t(M)\subset N$ for
all $t$), then $H\cdot X = H_N\cdot X$, and therefore
\begin{align*}
  \int_M \Div_M X \,dS 
  &= \int_{\partial M}X\cdot\nu_{\partial M}\,ds - \int_M X\cdot H\,dS 
  \\
  &= \int_{\partial M}X\cdot\nu_{\partial M}\,ds - \int_M X\cdot H_N\,dS,
\end{align*}
which is the generalized divergence theorem for $M$ as a submanifold of $N$.
\end{remark*}

\begin{quote}
{\bf Exercise}: Let $u:\RR^n\to \RR$ be a $C^1$ function.  Prove (under
the assumptions of theorem~\ref{first-variation-theorem}) that
\begin{align*}
\ddt_{t=0}\int_{\phi_tM}u\,dS
&=  \int_M \nabla u \cdot X\,dS + \int_M u\Div_MX\,dS \\
&=  \int_{\partial M}u\,X\cdot \nu_{\partial M}\,ds + \int_M ((\nabla u)^\perp-uH)\cdot X\,dS.
\end{align*}
where $(\nabla u)^\perp$ is the component of $\nabla u$ normal to $M$.
\end{quote}

\begin{definition*}
An $m$-dimensional submanifold $M\subset \RR^n$ (or of a Riemannian manifold)
 is called {\bf minimal} (or {\bf stationary})
provided its mean curvature is everywhere $0$, i.e., provided it is a critical point for the area functional.
\end{definition*}

\begin{theorem}\label{area-formula-theorem}
Let $M$ be a compact $m$-dimensional minimal submanifold of $\RR^n$. 
Then
\begin{equation}\label{area-formula-equation}
    m\, \area(M) = \int_{x\in \partial M} x\cdot \nu_{\partial M}\,ds.  
\end{equation}
\end{theorem}

\begin{proof}
We apply 
\[
     \int_M \Div_MX\,dS = \int_{\partial M} X\cdot \nu_{\partial M} - \int_M X\cdot H\,dS
\]
to the vectorfield $X(x)\equiv x$.  
Now $\Div_MX\equiv m$ and $H\equiv 0$, so we get~\eqref{area-formula-equation}.
(To see that $\Div_MX=m$, note that 
$\nabla_{\vv}X=\vv$ for every vector $\vv$, and thus that   
       $\sum_i(\ee_i\cdot\nabla_{\ee_i}X)=m$
for any orthonormal vectors $\ee_1,\dots,\ee_m$.)
\end{proof}

\begin{remark*}
Even if $M$ is not minimal, the same proof shows that 
\begin{equation}\label{general-area-formula-equation}
   m \area(M) = \int_{x \in \partial M} x\cdot \nu_{\partial M}\,ds - \int_{x\in M}x\cdot H\,dS.
\end{equation}
\end{remark*}

\section*{\quad Monotonicity}

\begin{theorem}[Monotonicity Theorem]\label{monotonicity-theorem}
Let $M$ be a minimal submanifold of $\RR^n$ and let $p\in \RR^n$.
Then
\[
   \Theta(M,p,r):=    \frac{\area(M\cap \BB(p,r))}{\omega_mr^m}
\]
is an increasing function of $r$ for $0<r\le R:=\dist(p,\partial M)$.

Indeed,
\[
    \ddr\Theta(M,p,r)\ge 0
\]
with equality if and only if $M$ intersects $\partial \BB(p,r)$ orthogonally.
\end{theorem}

Here $\omega_m$ is the $m$-dimensional area (i.e., Lebesgue measure)
of the unit ball in $\RR^m$.  Thus $\Theta(M,p,r)$ (which is called the {\bf density
ratio} of $M$ in $\BB(p,r)$) is the area of $M\cap\BB(p,r)$ divided by the
area of the cross-sectional $m$-disk in $\BB(p,r)$.  Equivalently, it is the
area of $M\cap \BB(p,r)$ after dilating by $1/r$.

\begin{proof}[Proof of monotonicity]
We may assume that $p=0$.  Let
$M_r = M\cap\BB(0,r)$, 
so 
\[    
      \partial M_r = M\cap \partial \BB(0,r).
\]
Let $A(r)$ be the $m$-dimensional area of $M_r$ and
$L(r)$ be the $(m-1)$-dimensional measure of $\partial M_r$. (When
 $m=2$, $A(r)$ is an area and $L(r)$ is a length.)

Then 
\[
   A'(r) \ge L(r). 
\]
This follows from the coarea formula applied to the function $x\in M\mapsto |x|$.
But intuitively (in the case $m=2$ for simplicity)  $A(r+dr)\setminus A(r)$ is 
a thin ribbon of surface: the length of the ribbon is $L(r)$ and the width is $\ge dr$.
(The width is equal to $dr$ at a point $p\in \partial M_r$ if and only if $M$ is 
orthogonal to $\partial \BB(0,r)$ at $p$.)  
Hence $A(r+dr)-A(r)\ge L(r)\,dr$.

By theorem~\ref{area-formula-theorem},
\[
   m A(r) =  \int_{\partial M_r} x \cdot \nu_{\partial M_r}\,ds \le r L(r).
\]
Combining these last two inequalities gives:
\[
  A' - mr^{-1}A \ge 0,
\]
so
\[
  r^{-m}A' - m r^{-m-1}A \ge 0,
\]
and therefore
\[
   (r^{-m}A)' \ge 0.  
\]
 \end{proof}

\begin{remark*}
The monotonicity theorem follows from the more
general monotonicity identity~\cite{simon-book}*{17.4}:
\begin{align*}
  \Theta(M,p,b) - \Theta(M,p,a)
  =
  &\int_{M(a,b)}\frac{|(\nabla r)^\perp|^2}{r^m}\,dS \\
  &+
  \frac1m\int_{x\in M(a,b)} \left( \frac1{r^m}-\frac1{b^m}\right) (x-p)\cdot H\,dS \\
  &+
  \frac1m\int_{x\in M(0,a)} \left( \frac1{a^m}-\frac1{b^m}\right) (x-p)\cdot H\,dS,
\end{align*}
where $0<a<b\le \dist(p,\partial M)$, where $r(\cdot)=\dist(\cdot,p)=|(\cdot)-p|$,
and where $M(a,b)=M\cap (\BB(p,b)\setminus \BB(p,a))$.
Note that if $M$ is minimal, then the last two integrals vanish.
One can use the monotonicity identity to prove a version of the monotonicity
theorem in a general Riemannian manifold $N$ by embedding $N$ isometrically
in $\RR^n$.  (If $M$ is a minimal submanifold
of $N\subset \RR^n$, then $M$ has locally
bounded mean curvature as a submanifold of $\RR^n$.) 
See~\cite{simon-book}*{17.6}, for example.
\end{remark*}

We define the {\bf density} of $M$ at a point $p\in M\setminus\partial M$ to be
\[
  \Theta(M,p):= \lim_{r\to 0} \Theta(M,p,r).
\]
For a smooth, immersed surface, the density of $M$ at a point $p\in M\setminus\partial M$
is equal to the number of sheets of $M$ that pass through $p$.
In particular, $\Theta(M,p)\ge 1$.

\section*{\quad Density at infinity}

Let $M$ be a properly immersed minimal surface without boundary in $\RR^n$.
Then $\Theta(M,p,r)$ is increasing for $0<r<\infty$.  Thus $\lim_{r\to\infty}\Theta(M,p,r)$
exists. (It may be infinite.)  Note that
\[
    \BB(p,r)\subset \BB(q, r + |p-q|)
\]
from which it easily follows that $\lim_{r\to \infty}\Theta(M,p,r)$ is independent of $p$
and therefore can be written without ambiguity as $\Theta(M)$.  We call $\Theta(M)$
{\bf the density of $M$ at infinity}.

For example, the density at infinity of a plane is $1$, and the density at infinity
of a union of $k$ planes is $k$.  
Near infinity, a catenoid (figure~\ref{catenoid-scherk}) looks like a multiplicity~$2$ plane.
(To be precise, if we dilate the catenoid by $1/n$ about its center and let $n\to\infty$,
then the resulting surfaces converge smoothly (away from the center) to 
a plane with multiplicity $2$.)
It follows that the catenoid has density $2$ at infinity.

Similarly, Scherk's surface (figure~\ref{catenoid-scherk}) 
resembles two orthogonal planes near infinity,
so its density at infinity is also $2$.

%

\begin{figure}
        \centering
        \begin{subfigure}[b]{0.5\textwidth}
                \centering
                \includegraphics[width=\textwidth]{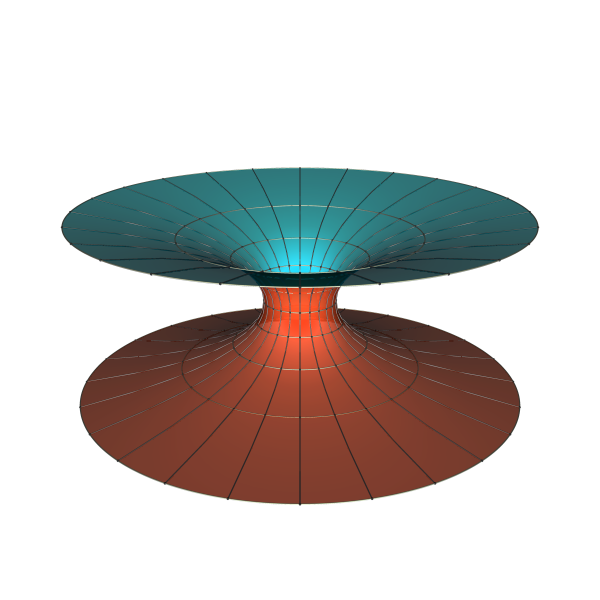}  
                \label{fig:gull}
        \end{subfigure}%
        ~ 
        \begin{subfigure}[b]{0.5\textwidth}
                \centering
                \includegraphics[width=\textwidth]{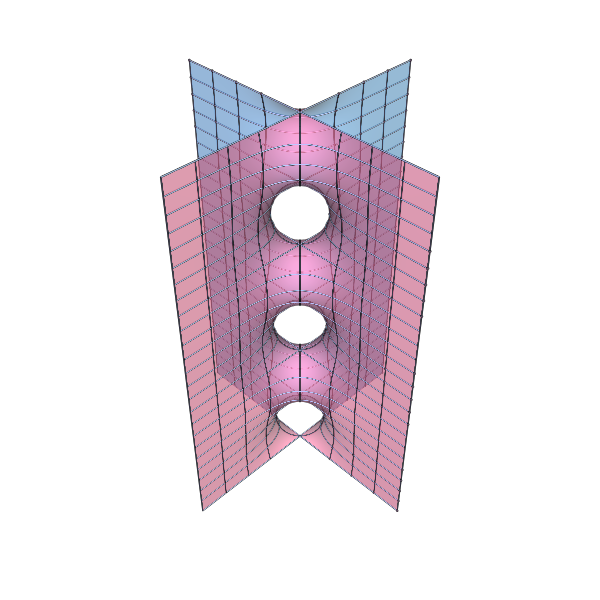}    
                \label{fig:tiger}
        \end{subfigure}
        \caption{The catenoid (left) and Scherks's surface (right) each have density $2$ at infinity.}
        \label{catenoid-scherk}
\end{figure}


The following theorem characterizes the plane by its density at infinity:

\begin{theorem}\label{plane-density-one-theorem}
Let $M$ be a properly immersed minimal $m$-manifold without boundary in $\RR^n$.
Then $\Theta(M)\ge 1$, with equality if and only if $M$ is a multiplicity $1$ plane.
\end{theorem}

\begin{proof} Let $p\in M$.  Then by monotonicity,
\begin{equation}\label{one-inequality}
  1 \le \Theta(M,p,r) \le \Theta(M). 
\end{equation}
This proves the inequality.  
If $1=\Theta(M)$, then we would have equality in~\eqref{one-inequality}, 
so $M$ would intersect $\partial \BB(p,r)$ orthogonally for every $r$.
That implies that $M$ is invariant under dilations about $p$, i.e., that
$M$ is a cone with vertex $p$.  Since we are assuming that $M$ is smooth,
$M$ must in fact be a union of planes (with multiplicity) passing through $p$.  
Since $\Theta(M)=1$,
$M$ is a single plane with multiplicity $1$.
\end{proof}

\section*{\quad Extended monotonicity} 

According to the monotonicity theorem (theorem~\ref{monotonicity-theorem}),
if $M$ is minimal, then the density ratio
\begin{equation}\label{density-ratio-equation}
   \Theta(M,p,r) = \frac{\area(M\cap\BB(p,r))}{\omega_mr^m}
\end{equation}
is an increasing function of $r$ for $0<r < R=\dist(p, \partial M)$.
The theorem is false without the restriction $r<\dist(p,\partial M)$.  
For example, if $M\subset \BB(p,\hat{R})$,
 then the density ratio
  is strictly {\em decreasing}
for $r\ge \hat{R}$, because the numerator of the fraction~\eqref{density-ratio-equation} is constant for $r\ge \hat{R}$.

However, there is an extension  of the monotonicity theorem  that gives information
for all $r$:

\begin{theorem}[Extended Monotonicity Theorem~\cite{EWW}]\label{extended-monotonicity-theorem}
Suppose that $M\subset \RR^n$ is a compact, minimal $m$-manifold 
with boundary $\Gamma$,
and that $p\in \RR^n\setminus \Gamma$.
Let $E=E(p,\Gamma)$ denote the {\bf exterior cone} with vertex $p$ over $\Gamma$:
\[
   E = \cup_{q\in \Gamma} \{ p + t (q-p): t\ge 1\}.
\]
Let $\tilde M= M\cup E$.  Then the density ratio
\[
     \Theta(\tilde M, p, r) := \frac{\area(\tilde M\cap \BB(p,r))}{\omega_mr^m}
\]
is an increasing function of $r$ for all $r>0$.  Indeed, 
\[
    \ddr\Theta(\tilde M, p,r)\ge 0,
\]
with equality if and only if: (i) $\nu_{\partial M}+\nu_{\partial E}\equiv 0$ on 
   $\Gamma\cap \BB(p,r)$ and (ii)  $\tM$ intersects 
 $\partial \BB(p,r)$ orthogonally.
\end{theorem}

\begin{remark*}
In the definition of $\Theta(\tilde M, p,r)$, we count area with multiplicity.
For example, if exactly two portions of $E$ overlap in a region, we  count the area of 
that region twice.
Of course if $M$ is embedded and if $p$ is in general position, then such overlaps do not occur.
\end{remark*}

In proving the extended monotonicity theorem, 
we may assume that $p=0$.  As before, we will apply the first variation formula (or, more
precisely, the generalized divergence theorem) to the vectorfield $X(x)=x$.

\begin{lemma}\label{unit-vectors-lemma}
 Let $E$ be the exterior cone over $\Gamma$ with vertex $0$.
Among all unit vectors $\vv$ that are normal to $\Gamma$ at $x\in \Gamma$,
the maximum value of $x\cdot \vv$ is attained by $\vv=-\nu_{\partial E}(x)$.
\end{lemma}
 
Consequently, $x\cdot\vv \le -x\cdot \nu_{\partial E}(x)$ 
and therefore $x\cdot (\vv + \nu_{\partial E}(x)) \le 0$ for every such vector $\vv$.
The proof of the lemma is left as an exercise.

\begin{proof}[Proof of extended monotonicity]
Let $M_r$, $E_r$, $\tM_r$, and $\Gamma_r$ be the portions of $M$, $E$, $\tM$, and
$\Gamma$ inside the ball $\BB_r=\BB(0,r)$.
By the generalized divergence theorem,
\begin{equation}\label{M-part-equation}
\begin{aligned}
  \int_{M_r} \Div_MX \, dS 
  &= 
  \int_{\partial M_r} X\cdot \nu_{\partial M_r}\,ds - \int_{M_r}H\cdot X\, dS  
  \\
  &=
  \int_{\partial M_r}X\cdot \nu_{\partial M_r}\,ds
\end{aligned}
\end{equation}
since $H\equiv 0$ on $M$.  Similarly,
\begin{equation}\label{E-part-equation}
\begin{aligned}
  \int_{E_r} \Div_EX \, dS 
  &= 
  \int_{\partial E_r} X\cdot \nu_{\partial E_r}\,ds - \int_{E_r}H\cdot X\, dS  
  \\
  &=
  \int_{\partial E_r}X\cdot \nu_{\partial E_r}\,ds
\end{aligned}
\end{equation}
because $H\cdot X\equiv 0$ on $E$, since $H$ is perpendicular to $E$ and $X$ is tangent to $E$.
Also, $\Div_MX\equiv \Div_EX\equiv m$, so the left sides of these equations
are $m\,\area(M_r)$ and $m\,\area(E_r)$.  
Adding equations~\eqref{M-part-equation} and~\eqref{E-part-equation} gives
\begin{equation}\label{both-parts-equation}
  m\, \area({\tilde M}_r)  
  \le 
  \int_{\partial M_r}x\cdot \nu_{\partial M_r}\,dS + \int_{\partial E_r}x\cdot \nu_{\partial E_r}\,dS.
\end{equation}
Note that $\partial M_r$ consists of two parts: $M\cap\partial \BB_r$ and $\Gamma_r$.
Likewise $\partial E_r$ consists of $E\cap \partial \BB_r$ and $\Gamma_r$.
By combining the two integrals over $M\cap \partial\BB_r$ and $E\cap\partial\BB_r$, and by also
combining the two integrals over $\Gamma_r$, 
we can rewrite~\eqref{both-parts-equation} as
\[
m\,\area(\tM_r)
\le
\int_{\partial {\tM}_r} x\cdot \nu_{\partial {\tM}_r}\,ds
+
\int_{\Gamma_r}  x\cdot  (\nu_{\partial M_r} + \nu_{\partial E_r})\,ds. 
\]
By lemma~\ref{unit-vectors-lemma}, the second integrand is everywhere nonpositive.  Thus
\[
  m\,\area(\tM_r)  \le \int_{x\in\partial \tM_r} x\cdot \nu_{\partial \tM_r}\,ds.
\]
The rest of the proof is exactly the same as the proof of the monotonicity theorem.
\end{proof}

\begin{corollary}\label{star-shaped-corollary}
Let $p\in M$, $\Gamma$, and $E=E(p,\Gamma)$ be as in the extended monotonicity theorem.
If $p\in M\setminus \Gamma$, then $1\le \Theta(M\cup E)$, with equality if and only
if $M\cup E$ is a multiplicity $1$ plane, i.e., if only if $M$ is a star-shaped region (with respect
to $p$) in an $m$-dimensional plane.
\end{corollary}

\begin{proof} The proof is almost identical to the proof of theorem~\ref{plane-density-one-theorem}.
\end{proof}

As a consequence of the extended monotonicity theorem, 
we can show that a minimal surface must stay reasonably close to 
its boundary:

\begin{theorem}\label{stay-close-theorem}
Let $M$ be a compact $m$-dimensional minimal submanifold of $\RR^n$.
Then for $p\in M$, 
\[  
  m\,\omega_m \dist(p, \partial M)^{m-1} \le |\partial M|.
\]
where $|\partial M|$ is the $(m-1)$-dimensional measure of $\partial M$.
Furthermore, 
equality holds if and only if $M$ is a flat $m$-disk centered at $p$ with multiplicity $1$.
\end{theorem}

\begin{proof}
We may assume that $p=0$.  Let $\Gamma=\partial M$, let $C$ 
be the (entire) cone over $\Gamma$:
\[
    \{t q: t\ge 0, \, q\in \Gamma\}
\]
and let $E= \{tq: t\ge 1, \, q\in \Gamma\}$ be the exterior cone. 
 Let $\Gamma^*$ be the result
of radially projecting $\Gamma$ to $\partial \BB(0,R)$, where $R=\dist(p,\partial M)$.
Then by extended monotonicity,
\begin{equation}\label{long-equation}
1
\le \Theta(M\cup E)  
=
\Theta(C) 
=
\Theta(C,0,R) 
=
\frac{1}{\omega R^m}\,\frac{R}m  \left| \Gamma^* \right| 
\le
\frac{\left| \Gamma \right|}{m\,\omega_m R^{m-1}} ,
\end{equation}
which is the asserted inequality.
Equality of the first two terms in~\eqref{long-equation} implies that $M\cup E$ is a 
plane with multiplicity $1$, 
and equality of the last two terms implies
 that $\Gamma^*=\Gamma$, which implies that the function $\dist(\cdot, 0)$ is constant
on~$\Gamma$.
\end{proof}

The following corollary implies (for example) that two short curves bounding 
 a connected minimal surface
cannot be too far apart:

\begin{corollary}
If $M\subset \RR^n$ is a compact, connected $m$-dimensional minimal submanifold such that
$\Gamma$ is the union of two (not necessarily connected) components $\Gamma_1$ 
and $\Gamma_2$,
then
\[
    \dist(\Gamma_1, \Gamma_2) \le 2 \left( \frac{|\partial M|}{m\,\omega_m} \right)^{1/(m-1)}.
\]
\end{corollary}

\begin{proof}
Since the function 
 $\dist(\cdot,\Gamma_1)-\dist(\cdot,\Gamma_2)$ is negative on $\Gamma_1$
 and positive on $\Gamma_2$, there must be a point $p\in M$ at which it vanishes.
 (Here $\dist$ denotes
the straight line distance in $\RR^n$.)  Let 
\[
  R=\dist(p,\Gamma_1)=\dist(p,\Gamma_2)=\dist(p,\Gamma).
\]
By the triangle inequality, $\dist(\Gamma_1,\Gamma_2)\le 2R$, which
is at most
\[
2 \left( \frac{|\partial M|}{m\,\omega_m} \right)^{1/(m-1)}
\]
by theorem~\ref{stay-close-theorem}.
\end{proof}

In~\cite{EWW}, the extended monotonicity theorem was used to solve
a long-open problem in minimal surface theory: if $\Gamma\subset \RR^3$ is
a smooth, simple closed curve with total curvature at most $4\pi$, must 
the unique\footnote{Nitsche~\cite{nitsche}
 had proved that a curve of total curvature less than $4\pi$ bounds
a unique minimal disk, and that the disk is smoothly immersed.
Whether such a curve can bound a minimal surface of nonzero genus is an interesting
open question.
Such curves can bound minimal M\"obius strips~\cite{EWW}.
Nitsche's uniqueness theorem was extended to curves of total curvature at most $4\pi$ by
X.~Li and Jost~\cite{Li-Jost-Unique}.}
minimal immersed disk bounded by $\Gamma$ be embedded?
(The total curvature of a smooth curve is the integral of the norm of the 
curvature vector with respect to arclength.)

\begin{theorem}\cite{EWW}\label{EWW-theorem}
Let $M$ be an immersed 
minimal surface
(possibly with branch points\,\footnote{Branch points will be discussed in lecture~\ref{lecture2}.})
 in $\RR^n$ bounded by a smooth embedded curve
$\Gamma$ whose total curvature is at most $4\pi$.  Then $M$ is smoothly embedded
(without branch points).
\end{theorem}

\begin{proof}
For simplicity, we give the proof for unbranched surfaces and 
for curves of total curvature strictly less than $4\pi$,
and we prove only that $M\setminus \Gamma$ has no points of 
self-intersection.
Let $p$ be a point in $M\setminus \Gamma$.  Let $C$ and $E$ be the cone and
 the exterior cone over
$\Gamma$ with vertex $p$, as in the extended monotonicity
 theorem~\ref{extended-monotonicity-theorem}.   
It is left as an exercise to the reader to show that
\[
  \Theta(C) \le \frac1{2\pi}\left(\text{the total curvature of $\Gamma$}\right).
\]
(This is just a fact about the geometry of curves.)  Thus by hypothesis, $\Theta(C)< 2$.
If we let $\tilde M = M\cup E$ as in the 
Extended Monotonicty Theorem~\ref{extended-monotonicity-theorem}, then 
 $\Theta(\tilde M)=\Theta(C)<  2$, so
\[
  \Theta(M,p)=\Theta(\tilde M, p)\le \Theta(\tilde M) < 2.
\]
Since $\Theta(M,p)$ is the number of sheets of $M$ passing through $p$, we
see that only one sheet passes through $p$.  Since $p$ is an arbitrary 
point in $M\setminus \Gamma$, we are done.
\end{proof}

We have not yet discussed branch points, but exactly the same argument 
rules out interior branch points (i.e., branch points not in $\Gamma$): 
one only needs to know the fact
 that the density of a minimal surface at a branch point is at least $2$.
 That fact is easily proved using the Weierstrass Representation (for example), which
 will be discussed in Lecture 2.  A similar argument rules out branch points and
 self-intersections at the boundary.
 
\begin{corollary}[Farey-Milnor Theorem]\label{farey-milnor-corollary}
If $\Gamma$ is a smooth, simple closed curve in $\RR^3$ with total curvature at most $4\pi$,
then $\Gamma$ is unknotted.
\end{corollary}

\begin{proof}
Let $M$ be a least-area disk bounded by $\Gamma$. (The disk exists by 
the Douglas-Rado Theorem, which will be discussed
in Lecture 4.)  By theorem~\ref{EWW-theorem}, $M$ is a smooth embedded disk.
But for any smoothly embedded curve, 
bounding a smooth embedded disk implies (indeed, is equivalent to)
being unknotted.
(To see the implication, let $F:\overline{D}\to M$ be a smooth, conformal parametrization of $M$
by the unit disk in $\RR^2$.  Then 
\begin{align*}
&h: (0,1]\times \partial D\to \RR^3, \\
&h: (t, p) = F(tp)
\end{align*}
provides an isotopy from $\Gamma$ to small, nearly circular curves, which are
clearly unknotted.)
\end{proof}

Theorem~\ref{EWW-theorem}
is sharp: there exist smooth embedded curves, including unknotted ones,
that have total curvature slightly larger than $4\pi$ and that bound many 
immersed minimal surfaces (see theorem~\ref{unknotted-alpha-theorem}).  
The Farey-Milnor Theorem is also sharp: consider, for example, a trefoil knot that
is a slight perturbation of a twice-traversed circle.

One can also define total curvature for arbitrary continuous curves.
Theorem~\ref{EWW-theorem} and corollary~\ref{farey-milnor-corollary} remain true for continuous simple closed curves with total curvature
at most $4\pi$. (The surface $M$ will be embedded, though of course it will in general
be smoothly embedded only away from its boundary.)   See~\cite{EWW}.

\section*{\quad The isoperimetric inequality}

The following fundamental theorem was proved by Allard~\cite{allard} 
(with a constant that was allowed to depend on dimension $n$ of the ambient space)
and by Michael and Simon~\cite{michael-simon}:

\begin{theorem}[Isoperimetric inequality]
Let $M$ be a smooth, compact $m$-dimensional submanifold of $\RR^n$.
Then
\[
   |M| \le c_m \left( |\partial M| + \int_M |H|\,dS \right)^{m/(m-1)}
\]
where $M$ is the $m$-dimensional measure of $M$ and 
$|\partial M|$ is the $(m-1)$-dimensional measure of $\partial M$.
\end{theorem}

See \cite{colding-minicozzi}*{3\S2}, \cite{simon-book}*{\S18}, or the original papers
for the proof.\footnote{In some of the references, the inequality is stated as a Sobolev
inequality for a function $u$ that is compactly supported on $M\setminus \partial M$.
The isoperimetric inequality follows by taking a suitable sequence of such $u$'s that 
converge to $1$ on $M\setminus \partial M$.}

\begin{quote}
{\bf Exercise}:  Prove the isoperimetric inequality for  a two-dimensional surface
with connected boundary.  (Use theorem~\ref{area-formula-theorem}.)
\end{quote}

The value of the best constant in the isoperimetric inequality, 
even in the case $H\equiv 0$ of minimal surfaces,
is an interesting open problem.  For minimal surfaces, it is conjectured that
the best constant is attained by a ball in an $m$-dimensional plane.
Almgren \cite{almgren-isoperimetric} proved the conjecture assuming that $M$ is area-minimizing.
For two-dimensional minimal surfaces, the conjecture has been proved
in some cases, such as when $\partial M$ has at most two 
connected components~\cite{li-schoen-yau}.
See~\cite{choe-schoen} for some more recent developments.


\lecture{Two-Dimensional Minimal Surfaces}\label{lecture2}

The theory of two-dimensional surfaces has many features that do 
not generalize to higher dimensional manifolds.
For example, every two-dimensional surface 
with a smooth Riemannian metric admits local isothermal coordinates,
i.e., can be parametrized locally by conformal maps from domains in the 
plane.\footnote{Existence of isothermal coordinates was proved by Gauss
for analytic metrics in 1822, and by Korn and by Lichtenstein for $C^{1,\alpha}$
metrics in 1916.  A number of other proofs have been given subsequently.
See, for example, lemma 1 and the paragraph following its proof in~\cite{DK}.}

\section*{\quad Relation to harmonic maps}

\begin{theorem}
Let $F:\Omega\subset \RR^2\to \RR^n$ be a conformal immersion.
Then $F(\Omega)$ is minimal if and only if $F$ is harmonic.
\end{theorem}

\begin{proof}
One way to show this is to calculate that the mean curvature $H$ is equal to the 
Laplacian $\Delta_gF$
of $F$ with respect to the pullback by $F$ of the metric on $\RR^n$.
(This is also true for immersions of $m$-dimensional manifolds $M$ into
 general Riemannian manifolds.)  
Thus $M$ is minimal
if and only if $F$ is harmonic with respect to the metric $g$.
The theorem follows immediately because, for two-dimensional surfaces, harmonic functions remain harmonic
under conformal change of metric on the domain.
\end{proof}

\begin{corollary}
Every two-dimensional $C^2$ minimal surface in $\RR^n$ is real-analytic.
\end{corollary}

This is also true for $m$-dimensional minimal submanifolds of $\RR^n$, but by a different proof.

\begin{theorem}[Convex hull property]\label{convex-hull-theorem}
Let $M$ be a two-dimensional minimal surface in $\RR^n$.
\begin{enumerate}
\item If $\phi: \RR^n\to \RR$ is a $C^2$ function, then $\phi|M$ cannot have
 a local maximum at any point of $M\setminus\partial M$ where $D^2\phi$ is positive definite.
\item If $M$ is compact, then it lies in the convex hull of its boundary.
\end{enumerate}
\end{theorem}

\begin{proof}
Let $p\in M\setminus \partial M$ be a point at which $D^2\phi$ is positive definite.
Let $F:D\subset \RR^2\to M$ with $F(0)=p$ be a conformal (and therefore harmonic)
parametrization of a neighborhood of $M$. 
Then (using $y_1,\dots, y_n$ 
and $x_1, x_2$ as coordinates for $\RR^n$ and $\RR^2$ and summing over 
repeated indices) one readily calculates by the chain rule that
\begin{equation}\label{subharmonic-equation}
  \Delta (\phi\circ F)
  = 
  \frac{\partial \phi}{\partial y_i} \cdot \Delta F_i 
  + 
  \frac{\partial^2\phi}{\partial y_i\,\partial y_j}\frac{\partial F_i}{\partial x_k}\frac{\partial F_j}{\partial x_k}
  \ge \lambda |DF|^2
\end{equation}
since $F$ is harmonic, where $\lambda$ is the lowest eigenvalue of $D^2\phi$.
This  is strictly positive at a point where $D^2\phi$ is positive definite.
Consequently $(\partial/\partial x_k)^2(\phi\circ F)$ must be positive for $k=1$ and/or $k=2$,
which proves (1).

To prove (2),
it suffices to show that that if $\partial M$ lies in a closed ball, then so does $M$,
since the convex hull of $\partial M$ is the intersection of all such balls.
If this failed for some ball $\BB(p,r)$, then the function $x\in M \mapsto |x-p|^2$
would attain its maximum at an interior point of $M$, contradicting (1).
\end{proof}

Theorem~\ref{convex-hull-theorem}  is also true for  
  $m$-dimensional minimal submanifolds of $\RR^n$ by essentially the same
  proof. (In particular,~\eqref{subharmonic-equation}
   is true at a point if $\Delta$ denotes the Laplacian with
  respect to the metric on $M$ induced from $\RR^n$ and if $x_1,\dots, x_m$
  are normal coordinates at that point.)  
  Theorem~\ref{convex-hull-theorem} 
  is a special case of much more general maximum principles for 
  (possibly singular) minimal varieties. See for example~\cite{white-maximum}.
  
{\bf Exercises}: 
\begin{enumerate}
\item\label{simply-connected-exercise}
 Suppose that $M$ is a compact, simply connected, two-dimensional minimal surface
in $\RR^n$ and that $\BB\subset \RR^n$ is a ball. Prove that $M\cap \BB$ is also 
simply connected.
\item\label{average-eigenvalue-exercise}
 In theorem~\ref{convex-hull-theorem}, show that $\phi|M$ cannot have
 a local maximum at any point of $M\setminus \partial M$ 
 at which the average of the two smallest eigenvalues of $D^2\phi$ is strictly
 positive.  (Hint: show that the inequality~\eqref{subharmonic-equation}
  for harmonic, conformal maps $F$
 remains true if we let $\lambda$ denote the average of the smallest two eigenvalues of $D^2\phi$.)
 \item\label{strictly-mean-convex-exercise}
  Show that if $M\subset \RR^3$ is a connected minimal surface
contained in a region $N$ with smooth boundary, then
$M\setminus \partial M$
cannot touch $\partial N$ at any point where the mean curvature of $\partial N$
is nonzero and points into $N$. 
(Hint: Let $\phi$ be the signed distance function to $\partial N$ such that $\phi<0$
 in the interior of $N$.)
\item (For readers familiar with the strong maximum
principle in partial differential equations.) Show in exercise~\eqref{strictly-mean-convex-exercise}
  that if the mean curvature
vector of $\partial N$ is a nonnegative multiple of the inward unit normal at all points
and that if $M\setminus \partial M$ touches $\partial N$, then $M$ is contained in 
  $\partial N$.
\end{enumerate}
All four exercises generalize to $m$-dimensional minimal submanifolds. 
(In exercise~\eqref{simply-connected-exercise}, replace ``simply connected" by 
``having trivial $(m-1)^{\rm th}$ homology".
In exercise~\eqref{average-eigenvalue-exercise}, 
replace ``smallest two eigenvalues" by ``smallest $m$ eigenvalues".)
Also, theorem~\ref{convex-hull-theorem} and the exercises remain true
for branched minimal surfaces (which will be discussed in lecture~\ref{lecture2}).

\section*{\quad Conformality of the Gauss map}

Recall that if $M$ is an oriented surface in $\RR^3$, then its Gauss map $\nn:M\to \Ss^2$
is the map that maps each point in $M$ to the unit normal to $M$ at that point.

\begin{theorem}\label{conformal-gauss-map-theorem}
Let $M$ be a minimal surface in $\RR^3$.
Then $M$ is minimal if and only if the Gauss map $\nn: M\to \Ss^2$ is almost conformal 
and orientation-reversing. 
\end{theorem}

Recall that a map $F:M\to N$ between Riemannian manifolds is almost conformal provided
\[
   DF(p)\uu \cdot DF(p)\vv \equiv \lambda(p) \uu\cdot \vv 
\]
for all $\uu,\vv\in \Tan_pM$, where $\lambda$ is a non-negative function.  The map
is conformal if $\lambda$ is everywhere positive.

In theorem~\ref{conformal-gauss-map-theorem}, 
the orientation on $\Ss^2$ is the standard orientation, i.e., the orientation
given by the outward unit normal.

\begin{proof}
Let $\ee_1$ and $\ee_2$ be the principal directions of $M$ at $p\in M$. Then $\{\ee_1, \ee_2\}$
is an orthonormal basis for $\Tan_pM$ and also for $\Tan_{\nn(p)}\Ss^2$.
With respect to this basis, the matrix for $D\nn(p)$ is 
\[
\begin{bmatrix}
\kappa_1 &0 \\
0 &\kappa_2 
\end{bmatrix}
\]
where $\kappa_1$ and $\kappa_2$ are the principal curvatures.
\end{proof}

\section*{\quad Total curvature}

For any surface $M\subset \RR^3$, the Gauss curvature $K=\kappa_1\kappa_2$ is the
signed Jacobian of the Gauss map.  The {\bf total absolute curvature}
(which we will call the total curvature, for short) of $M$ is
\[
   TC(M) = \int_M |K|\, dS,
\]
which is equal to the area (counting multiplicity) of the image of $M$ under the Gauss map:
\[
  TC(M) = \area(\nn,M) = \int_{p\in \Ss^2} \# \nn^{-1}(p)\, dp.
\]
If $M$ is minimal, $K\le 0$, so
\[
   TC(M) = \int_M |K|\, dS = -\int_M K\,dS. 
\]

\begin{theorem}[Osserman~\cite{osserman-FTC}]
Let $M\subset \RR^3$ be a complete, connected, orientable minimal surface of finite total curvature:
\[
 TC(M)= \int_M|K|\,dS = \int_M(-K)\, dS < \infty.
\]
Then
\begin{enumerate}

\item $M$ is conformally equivalent to a compact Riemann surface minus finitely many points:
\[   
   M \cong \Sigma \setminus \{p_1, \dots, p_k\}.
\]

\item The Gauss map extends analytically to the punctures.
\item\label{preimage-number}
 There is a nonnegative integer $m$ such that for almost every $v\in \Ss^2$,
  exactly $m$ points in $M$ have unit normal $\nn=v$. 
\item The total curvature of $M$ is a equal to $4\pi m$.

\item $M$ is proper.
\end{enumerate}
\end{theorem}

Properness means that if we go off to infinity
in $M$, we also go off to infinity in $\RR^n$.  More precisely, we say that a sequence
$p_i\in M$ diverges in $M$ if no subsequence converges with
respect to the induced arclength metric on $M$ to a point in $M$.
We say that a sequence diverges in $\RR^n$ if  no 
subsequence converges with respect to the metric on $\RR^n$.
We say that $M$ is a proper in $\RR^n$ if every sequence $p_i\in M$
that diverges in $M$ also diverges in $\RR^n$.

For example, the curve $C=\{(x,\sin(\pi/x):x >0 \}$ is not  proper in $\RR^2$:
the sequence $(1/n,0)\in C$ diverges in $C$ but converges in $\RR^2$.
(The disconnected curve $C\cup Y$, where $Y$ is the $y$-axis, also is
not a proper submanifold of $\RR^2$ for the same reason.)

The catenoid $M$ with a vertical axis of rotational symmetry (figure~\ref{catenoid-figure}) 
provides a good example of Osserman's theorem.
The Gauss map is a conformal diffeomorphism from $M$ to $\Ss^2\setminus\{NP, SP\}$,
where $NP=(0,0,1)$ and $SP=(0,0,-1)$ are the north and south poles on $\Ss^2$.
In particular, $M$ is conformally diffeomorphic (in this case by the map $\nn$)
 to a twice-punctured sphere.
The Gauss map extends continuously to the punctures.  The total curvature
is the area of the Gaussian image, namely $4\pi$.

\begin{figure}
{\includegraphics[width=3in]{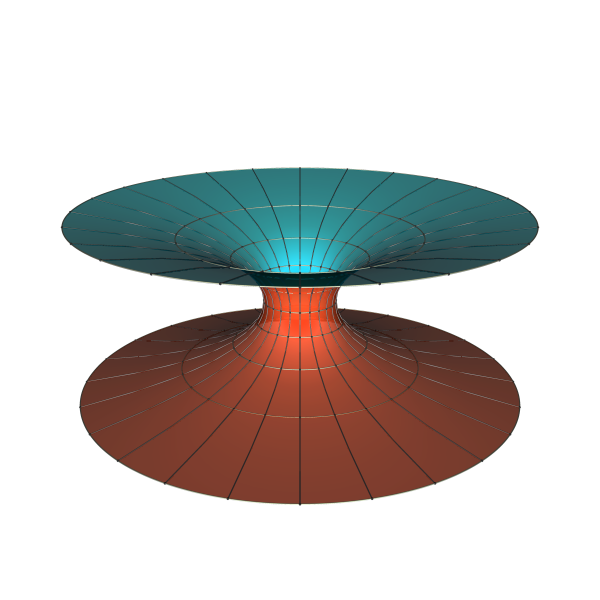}}\label{catenoid-figure}   
\caption{The catenoid has total curvature $4\pi$.}\label{catenoid-figure}
\end{figure}

\begin{proof}[Proof of Osserman's Theorem]
The first assertion is a special case of an intrinsic theorem due to Huber~\cite{huber}:
if $M$ is a complete, connected surface such that
\[
      \int_M K^- \,dS < \infty
\]
then $M$ is conformally a punctured Riemann surface. Here
\[
  K^- =
  \begin{cases}
  |K| &\text{if $K\le 0$}\, \\
  0  &\text{if $K>0$}.
  \end{cases}
\]

To prove the second assertion, let us (for the moment) orient $\Ss^2$ by the
inward-pointing unit normal, so that the Gauss map becomes orientation
preserving.
Let $U\subset \Sigma$ be a neighborhood of one of the punctures, $p$. 
By Picard's Theorem, either $\nn:U\setminus\{p\}\to \Ss^2\cong \CC\cup\{\infty\}$ is meromorphic 
at $p$ (and therefore extends to $p$), or $\nn:U\setminus\{p\} \to \Ss^2$ takes
all but two values in $\Ss^2$ infinitely many times.   
The latter implies that $\int_U|K|\,dS=\infty$, a contradiction.
Thus $\nn$ extends continuously (indeed analytically) to $U$.

We have shown that the Gauss map $\nn:\Sigma\to \Ss^2$ 
is holomorphic (with respect to orientation on $\Ss^2$
induced by the inward unit normal.)  Let $m$ be its mapping degree.
Then assertion~\eqref{preimage-number} holds by standard complex analysis or differential 
topology.

The fourth assertion (that the total curvature is $4\pi m$) follows immediately,
since
\[
  TC(M)= \int_{\Ss^2} \# \nn^{-1}(\cdot)\,dS.
\]

The last assertion (properness) 
can be proved using the Weierstrass 
Representation
(discussed below).
Alternatively, one can show that if $S\subset \RR^3$ is a complete surface
diffeomorphic to a closed disk minus its center and if the slope of $\Tan(S,p)$
is uniformly bounded, then $S$ is proper in $\RR^3$; see~\cite{white-complete}.
One applies this result to a small neighborhood in $\Sigma$ of a puncture, on which
one can assume (by rotating) that the unit normal is very nearly vertical. 
\end{proof}

\begin{remarks*} (1).  Every multiple of $4\pi$ does occur as the total curvature of such a surface.
(2).  For a proof of generalization of  Osserman's Theorem that does not use Huber's Theorem,
see~\cite{white-complete}.
\end{remarks*}

Osserman's Theorem has an extension, due to Chern and Osserman~\cite{chern-osserman},
 to two-dimensional
surfaces in $\RR^n$.  All the conclusions remain true, except that 
the total curvature\footnote{Here we can take the total curvature
to be the integral of the absolute value of the scalar curvature.  
Since the scalar curvature is everywhere nonpositive, this is equal to minus
the integral of the scalar curvature.  Since $M$ is minimal, the scalar curvature
is equal to $-\frac12|A|^2$, where $|A|$ is the norm of the second fundamental form,
so we could also define the total curvature
to be the integral of $\frac12|A|^2$.}
 is a multiple of $2\pi$ rather than of $4\pi$. 
And every multiple of $2\pi$ does occur as the total curvature of such a surface.
For example,
\[
\{ (z,w)\in \CC^2\cong \RR^4:  w=z^n \}.
\]
is a complete, embedded minimal surface with total curvature $2\pi(n-1)$.
The surface is minimal and indeed area minimizing by the Federer-Wirtinger
theorem (theorem~\ref{federer-wirtinger-theorem}).  To see that it has total curvature $2\pi(n-1)$,
let $M_r$ be the portion of the surface with $\{|w|\le r\}$.
Note that if $r$ is large, then $\partial M_r$ is very nearly a circle of radius $r$ traversed $n$ times.
Thus by the Gauss Bonnet theorem,
\[
    2\pi n \cong \int_{\partial M_r}k\,ds = 2\pi - \int_{M_r}K\,dS.
\]
Letting $r\to \infty$ gives $2\pi n = 2\pi + TC(M)$.

Theorem~\ref{plane-density-one-theorem}  characterized the plane by its
density at infinity.
Using Osserman's Theorem, we can give another characterization of the plane:

\begin{corollary}\label{four-pi-plane-corollary}
If $M\subset\RR^3$ is a complete, orientable minimal surface of total curvature $<4\pi$, 
then $M$ is a plane. 
\end{corollary}

\begin{proof}
If the total curvature is less than $4\pi$, it must be $0$, so $K\equiv 0$.
But for a minimal surface in $\RR^3$, $K(p)=0$ implies that the principal curvatures at $p$ are~$0$.
\end{proof}

Similarly, using the Chern-Osserman Theorem, one sees that a complete, orientable minimal
surface in $\RR^n$ with total curvature $<2\pi$ must be a plane.

As will be explained in lecture~\ref{lecture3}, corollary~\ref{four-pi-plane-corollary} implies a useful curvature 
estimate (theorem~\ref{four-pi-estimate-theorem}) for minimal surfaces.


\section*{\quad The Weierstrass Representation}\label{weierstrass-section}

We have seen that immersed minimal surfaces in $\RR^3$ are precisely those that can be parametrized
locally by conformal, harmonic maps $F: \Omega\subset \RR^2\to \RR^3$.
Following work of Riemann, 
Weierstrass 
and Enneper\footnote{In the interests of brevity,
I use the conventional name ``Weierstrass Representation''
rather than the more accurate ``Enneper-Weierstrass Representation".}
independently found a nice way to generate all such $F$.

Write $z=x+iy$ in $\RR^2$, and
\[
  \ddz =  \frac12\left(\ddx - i \ddy\right) \quad, \quad  \ddzbar  = \frac12\left(\ddx+i\ddy\right).
\]
Note that if $u$ is a map from $\Omega\subset \RR^2\cong\CC$ to $\CC$  (or more generally to $\CC^n$),
then $u_{\overline{z}}=0$ if and only if $u$ is holomorphic. (The real and imaginary parts 
of the equation 
 $u_{ \overline{z}}=0$ are precisely the Cauchy-Riemann equations for $u$.)
  
Note also that 
\[
  \ddzbar \ddz 
  = \frac14\left(\ddx+i\ddy\right)\left(\ddx- i\ddy\right) 
  = \frac14\left(\ddx^2+\ddy^2\right)
  =\frac14\Delta.
\]
Thus
\[
  \text{$F$ is harmonic} \iff \text{$F_{z\overline{z}} = 0$} \iff \text{$F_z$ is holomorphic}.
\]

Concerning conformality, note that if we extend the Euclidean inner product
from $\RR^n$ to $\CC^n$ by making it complex linear in both arguments, then
\begin{align*}
 F_z\cdot F_z 
 &= \frac14(F_x - i F_y)\cdot (F_x - i F_y)  \\
 &= \frac14\left( |F_x|^2 - |F_y|^2  - 2 i F_x\cdot F_y\right),
\end{align*}
so the real and imaginary parts of $F_z\cdot F_z$ are
\[
\Re(F_z\cdot F_z) = |F_x|^2 - |F_y|^2  \quad,\quad \Im(F_z\cdot F_z)= 2 F_x\cdot F_y.
\]
Consequently, $F$ is almost conformal if and only if $F_z\cdot F_z\equiv 0$.

We have proved:

\begin{theorem}\label{harmonic-conformal-theorem}
Let $F:\Omega\subset \CC\to \RR^n$ be a $C^2$ map.  Then $F$ is harmonic 
if and only if $F_z$ is holomorphic, and $F$ is almost conformal
if and only if $F_z\cdot F_z\equiv 0$.
\end{theorem}

Let $\phi = 2F_z = (\phi_1,\phi_2,\phi_3) \in \CC^3$.
Our goal is to find holomorphic $\phi$ such that $\phi\cdot \phi\equiv 0$.
We then recover $F$ by
\[
     F = \Re \left( \int 2F_z\, dz \right) = \Re \left( \int \phi\, dz\right).
\]

Note that $\phi(z) = 2F_z = F_x - i F_y$ determines the oriented tangent plane to $M$
at $F(z)$, and therefore the image $\nn(z)$ of $F(z)$ under the Gauss map, and thus
the image $g(z)\in \RR^2\cong \CC$ of $\nn(z)$ under stereographic projection.
Indeed, one can check that
\[
  g(z) = \frac{\phi_3}{\phi_1-i \phi_2}.
\]
One can solve for $\phi_1$ and $\phi_2$ in terms of $g$ and $\phi_3$:
\begin{align*}
   \phi_1 &= \frac12 (g^{-1}-g) \phi_3, \\
   \phi_2 &= \frac{i}2 (g^{-1}+g) \phi_3.
\end{align*}

\begin{theorem}[Weierstrass Representation]\label{weierstrass-theorem}
Let $\Omega\subset \CC$ be simply connected,  let $\phi_3$ and $g$
be a holomorphic function and a meromorphic function on $\Omega$, and let $\wform$
be the one form $\phi_3(z)\,dz$.  
 Suppose also that wherever $g$ has a pole  or zero of order $m$,
the function $\phi_3$ has a zero of order $\ge 2m$.  Then
\begin{equation} \label{weierstrass-equation}
   F(z)
   = \Re  \int^z \left( \frac12(g^{-1}-g), \frac{i}2(g^{-1}+g), 1 \right)  \wform 
\end{equation}   
is a harmonic, almost conformal mapping of $\Omega$ into $\RR^3$.

Furthermore, every
 harmonic, almost conformal map $F:\Omega\to\RR^3$ arises
in this way, unless the image of $F$ is a horizontal planar region.
\end{theorem}

The function $g$ may take the value $\infty$ because the surface may have
points where the unit normal is $(0,0,1)$.
  Thus $g$ may have poles.
  At the poles and zeroes of $g$, $F$ would have poles unless $\phi_3$
  has zeroes to counteract the poles of $g^{-1}\pm g$.
  Hence the condition about orders of poles and zeroes.

The ``unless" proviso is needed because if $F(\Omega)$ 
is a horizontal planar region, then $g\equiv 0$ or $g\equiv \infty$
and thus~\eqref{weierstrass-equation} does not make sense.

More generally, $\wform$ and $g$ can be any holomorphic differential and any meromorphic
function on any Riemann surface $\Omega$.  
 But if $\Omega$ is not simply connected,
the expression~\eqref{weierstrass-equation}
 may be well-defined only on the universal cover of $\Omega$.
To be precise, the Weierstrass representation \eqref{weierstrass-equation}
 gives a mapping of $\Omega$ (rather than of its
universal cover) if and only if the closed one forms
\[
   \frac12(g^{-1}-g)\wform, \,\,  \frac{i}2(g^{-1}+g)\wform , \,\,\text{and $\wform$}
\]
have no real periods.

\begin{figure}
        \centering
        \begin{subfigure}[b]{0.5\textwidth}
                \centering
                \includegraphics[width=\textwidth]{catenoid.png}  
                \label{fig:gull}
        \end{subfigure}%
        ~ 
        \begin{subfigure}[b]{0.5\textwidth}
                \centering
                \includegraphics[width=\textwidth]{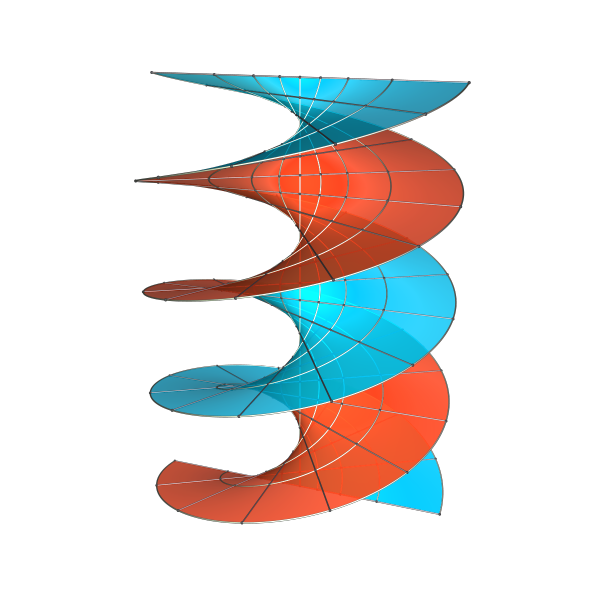}  
                \label{fig:tiger}
        \end{subfigure}
        \caption{The catenoid has 
        Weierstrass data $\wform(z)=dz/z$ and $g(z)=z$ on $\CC\setminus\{0\}$.
        The helicoid has Weierstrass data $\wform(z)=dz$ and $g(z)= e^{iz}$ on $\CC$.}
        \label{catenoid-helicoid}
\end{figure}

Since $g$ and $\phi_3$ (or $\wform$) determine the surface, all geometric quantities can be expressed
in terms of them.  
For example, the conformal factor $\lambda$ is given
by
\[
   \lambda= |F_x|=  |F_y|,
\]
so
\[
  |\phi| = \left| F_x - i F_y\right| =  \sqrt2 \lambda,
\]
and thus (calculating $|\phi|$ from~\eqref{weierstrass-equation}) we have
\begin{equation}\label{conformal-factor-equation}
  \lambda = \frac1{\sqrt{2}} |\phi| =   \frac{(|g|^{-1} + |g|)}2 |\phi_3|.
\end{equation}
That is,
\[
ds^2 
= \left[ \frac{(|g|^{-1} + |g|)}2 \right]^2 |\phi_3|^2\, |dz|^2   
= \left[ \frac{(|g|^{-1} + |g|)}2 \right]^2 |\wform|^2
\]
One can also calculate the Gauss curvature:
\[
    K 
    = -\left[  \frac{4|g'| }{|\phi_3 g| (|g|^{-2} + |g|^2)^2} \right]^2 
    = -\left[ \frac4{ (|g|^{-2} + |g|^2)^2} \left| \frac{dg}{g\wform} \right| \right]^2
\]
(Note: $dg$ and $g\wform$ are both meromorphic $1$-forms, so their ratio
is a meromorphic function.)
The points where $F$ fails to be conformal are called {\bf branch points}.
Using the expression~\eqref{conformal-factor-equation}
 for the conformal factor, we can identify the branch points:

\begin{proposition}
Suppose in theorem~\ref{weierstrass-theorem} that $g$ has a pole or zero of order $m\ge 0$ at $p$
and that $\phi_3$ has a zero of order $k\ge 2m$ at $p$.
Then $F$ is an immersion at $p$ if and only $k=2m$.  Thus $F$ has a branch point
at $p$ if and only if $k>2m$.
\end{proposition}

The difference $k-2m$ is called the {\bf order} of the branching at $p$.
It is not hard to show that for sufficiently small $\eps>0$, the density
of $F(\BB(p,\eps))$ at $F(p)$ is $1$ plus the order of branching at $p$.

\section*{\quad The geometric meaning of the Weierstrass data}

As explained above, the function $g$ in the Weierstrass representation has a simple geometric 
meaning: it is the Gauss map (or, more precisely, the Gauss map followed by stereographic
projection to $\RR^2\cup\{\infty\}\cong \CC\cup\{\infty\}$).

As for $\phi_3$, note by theorem~\ref{weierstrass-theorem} 
and the discussion preceding it
that
\[
 \phi_3 = 2\frac{\partial F_3}{\partial z} 
            = \frac{\partial F_3}{\partial x} - i\,\frac{\partial F_3}{\partial y},
\]
where $F=(F_1, F_2, F_3)$.   In other words, $\phi_3$ encodes the 
derivative of the height function $F_3$ with respect to the parametrization.
 More generally, as mentioned above, if the domain
of $F$ is a Riemann surface, we should think of the Weierstrass data as being
$g$ together with a holomorphic one form $\wform$ (corresponding to $\phi_3(z)\,dz$.)  
In this
case,
\[
  \wform = 2\,\partial F_3,
\]
where $\partial$ is the Dolbeault operator which (in any local holomorphic chart)
is given by $\partial(\cdot)= \ddz (\cdot)\,dz$.

It can be difficult to determine for a particular $g$ and $\phi_3$ (or $g$ and $\wform$)
 whether the real periods vanish
and (if they do vanish)  whether the resulting surface is embedded.  For those reasons, 
great ingenuity is often required to prove the existence of specific kinds of embedded, genus $g$
surfaces using the Weierstrass representation.

Part of the discussion above carries over without change to two-dimensional minimal surfaces
in $\RR^n$.  In particular, $F:\Omega\subset \RR^n$ is harmonic and almost conformal
if and only $F$ can be written as 
\[
   \Re \left( \int \phi\,dz \right)
\]
where $\phi=(\phi_1,\dots,\phi_n): \Omega\to \CC^n$ is a holomorphic map such that 
  $\phi\cdot \phi\equiv 0$.
Of course we can use the equation $\phi\cdot\phi\equiv 0$ to express $\phi$ in terms 
of $(n-1)$ holomorphic functions; those $(n-1)$ functions can then be chosen more-or-less
arbitrarily.  See~\cite{osserman-book}*{\S12} for more details.

As an example of how the Weierstrass representation gives nontrivial 
information about a surface, 
consider Enneper's surface $M$ (figure~\ref{enneper-figure}), i.e., the surface
with Weierstrass data $\eta(z)=\phi_3(z)\,dz=z\,dz$ and 
$g(z)=z$ on the entire plane.
As suggested by the picture, there is a finite group of congruences of $M$, i.e., 
of isometries of $\RR^3$
that map $M$ to itself.  Though it is not evident from the picture, in fact
$M$ has an infinite group of intrinsic isometries:   $M$
is intrinsically rotationally symmetric about $z=0$, since the metric
depends only on $|\phi_3|$ and $|g|$, which in this case are both equal to $|z|$.

\begin{figure}
{\includegraphics[width=2.in]{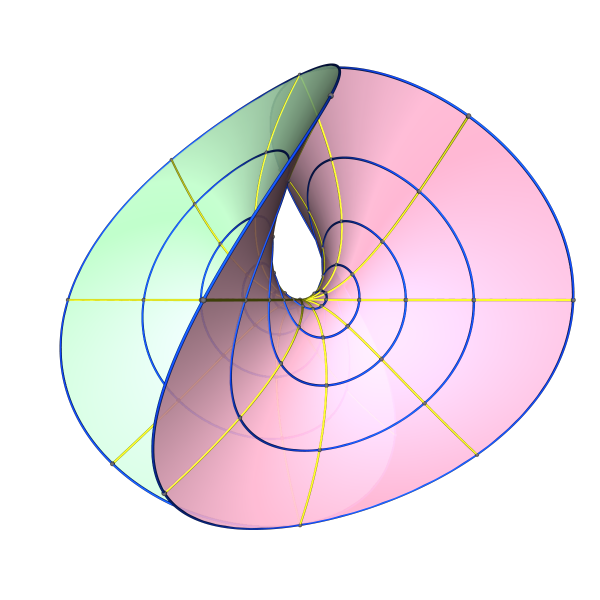}}  
\caption{Enneper's surface has Weierstrass data $\eta(z)=z\,dz$ and $g(z)=z$.
The surface is intrinsically rotationally symmetric about the point $z=0$.}
\label{enneper-figure}
\end{figure}

\section*{\quad Rigidity and Flexibility}

A surface $M\subset \RR^3$ is {\bf flexible} if it can be deformed (non-trivially)
through a one-parameter family of smooth isometric immersions.  Otherwise it is {\bf rigid}.
(Here ``trivially'' means ``by rigid motions of $\RR^3$''.)
For example, a flat rectangle (or, more generally, any planar domain)
  is flexible, as can be demonstrated by rolling up
a piece of paper.  
According to a classical theorem, every smooth, closed, convex surface
in $\RR^3$ is rigid.  
Whether there exists a flexible smooth, closed, non-convex surface $M$ is a long-open 
problem.

Is a round hemisphere rigid?
Perhaps there is no intuitive reason why it should or should not be rigid.
However, it (or, more generally, any proper closed subset of any closed, uniformly convex
surface) is flexible, 
as explained (if the closed surface is a sphere) in~\cite{HCV}*{\S32.10}.
On the other hand, in the surface $z=(x^2+y^2)^3$ (for example), arbitrarily small neighborhoods
of the origin are rigid~\cite{usmanov}.  
So whether particular surfaces are rigid or flexible can be rather subtle.

Now we impose an extra condition that makes flexing a surface much more difficult:
is there a non-trivial one-parameter family $F_t: M\to \RR^3$ of smooth isometric
immersions of a surface such that the unit normal at each point remains constant?
In other words, if $\nn(p,t)$ is the unit normal to $F_t(M)$ at $F_t(p)$, we require
that $\nn(p,t)$ be independent of $t$.   If $M_0=F_0(M)$ is a planar region, for example,
the answer is ``no", since in that case the only allowed deformations are translations and
rotations about  axes perpendicular to $M_0$.

Intuitively, non-trivial isometric deformations keeping the normals constant should
be impossible.  However, such deformation do exist
for  {\em every} simply connected, nonplanar minimal surface!
In the Weierstrass representation, replace $\wform$ by $e^{i\theta}\wform$ and
let $\theta$ vary.
The metric depends only on $|g|$ and $|\wform|$, so it does not change, 
and the Gauss map $g$ also does not change.  (The proof that the resulting deformations
are nontrivial for nonplanar $M$ is left as an exercise.)

For example, if we start with the helicoid, this gives an isometric deformation
of the helicoid to the catenoid (covered infinitely many times).
One may see animations of the deformation online, for example at 
\[
\text{\url{http://www.foundalis.com/mat/helicoid.htm}},
\]
\begin{quote}
\url{http://www.indiana.edu/~minimal/archive/Classical/Classical/AssociateCatenenoid/web/qt.mov},
\end{quote}
and
\begin{quote}
\url{http://virtualmathmuseum.org/Surface/helicoid-catenoid/helicoid-catenoid.mov}.
\end{quote}

Incidentally, one can show that minimal surfaces are the only surfaces in $\RR^3$
that can be isometrically deformed keeping the normals fixed, and that
the one-parameter family obtained by replacing $\wform$ by $e^{i\theta}\wform$ is the 
only such deformation of a minimal surface.

See~\cite{hoffman-karcher} or~\cite{weber} for more 
information about the Weierstrass Representation.  (Note: in those works,
the authors use $dh$ to denote the holomorphic one-form $\eta$.  
If one thinks of $h$ as the height function
on the surface, then their $dh$ is {\em not} the exterior derivative of $h$, but
rather $2\,\partial h$.  The exterior derivative of the height function is the real part of their $dh$.)

\lecture{Curvature Estimates and Compactness Theorems}\label{lecture3}
 
 In many situations, one wants to take limits of minimal surfaces.
For example, David Hoffman, Martin Traizet, and I needed to do this
in our recent work on genus-$g$ helicoids.  For several centuries, the plane and the helicoid were
the only known complete, properly embedded minimal surfaces in $\RR^3$
with finite genus and with exactly one 
end.\footnote{For a complete minimal surface $M$ properly immersed in $\RR^n$,
the number of ends is the limit as $r\to\infty$ of the number of connected components
of $M\setminus \BB(0,r)$.  It follows from the convex hull property 
(theorem~\ref{convex-hull-theorem}) that the number of those components cannot decrease 
as $r$ increases, and thus that the limit exists.}
Jacob Bernstein and Christine Breiner~\cite{bernstein-breiner}
proved (using work of Colding and Minicozzi)
that any such surface other than a plane must be asymptotic to a helicoid at infinity.
(Later Meeks and Per\'ez~\cite{meeks-perez-embedded}
 gave a different proof, also based on the work of Colding
and Minicozzi.)
Hence such a surface of genus $g$ is called an embedded {\bf genus-$g$ helicoid}.

But do embedded genus-$g$ helicoids exist for $g\ne 0$?
In~1992, Hoffman, Karcher, and Wei~\cite{HKW} used the Weierstrass Representation to 
prove the existence but not the embeddedness of a genus-one helicoid.
In 2004, Hoffman, Weber, and Wolf~\cite{HWW}  proved
existence of an embedded genus-$1$ helicoid as shown in figure~\ref{genus-one-figure}.
(See~\cite{HW} for a different, somewhat shorter proof.)
Although genus-$2$ examples were found numerically as early 
as 1993 (see figure~\ref{genus-two-figure}),
existence was not known rigorously until 2013, when Hoffman, Traizet,
and I~\cite{HTWa, HTWb}
proved existence of embedded genus-$g$ helicoids for every positive integer $g$.
In our proof,  first we construct analogous surfaces in $\Ss^2\times \RR$
(which, oddly enough, turns out to be easier), and then we get examples in $\RR^3$ by
letting the radius of the $\Ss^2$ tend to infinity.
Of course we need to know that the surfaces in $\Ss^2\times\RR$ 
converge smoothly (after passing to subsequences) to limit surfaces in $\RR^3$.

\begin{figure}[h]
{\includegraphics[width=2in]{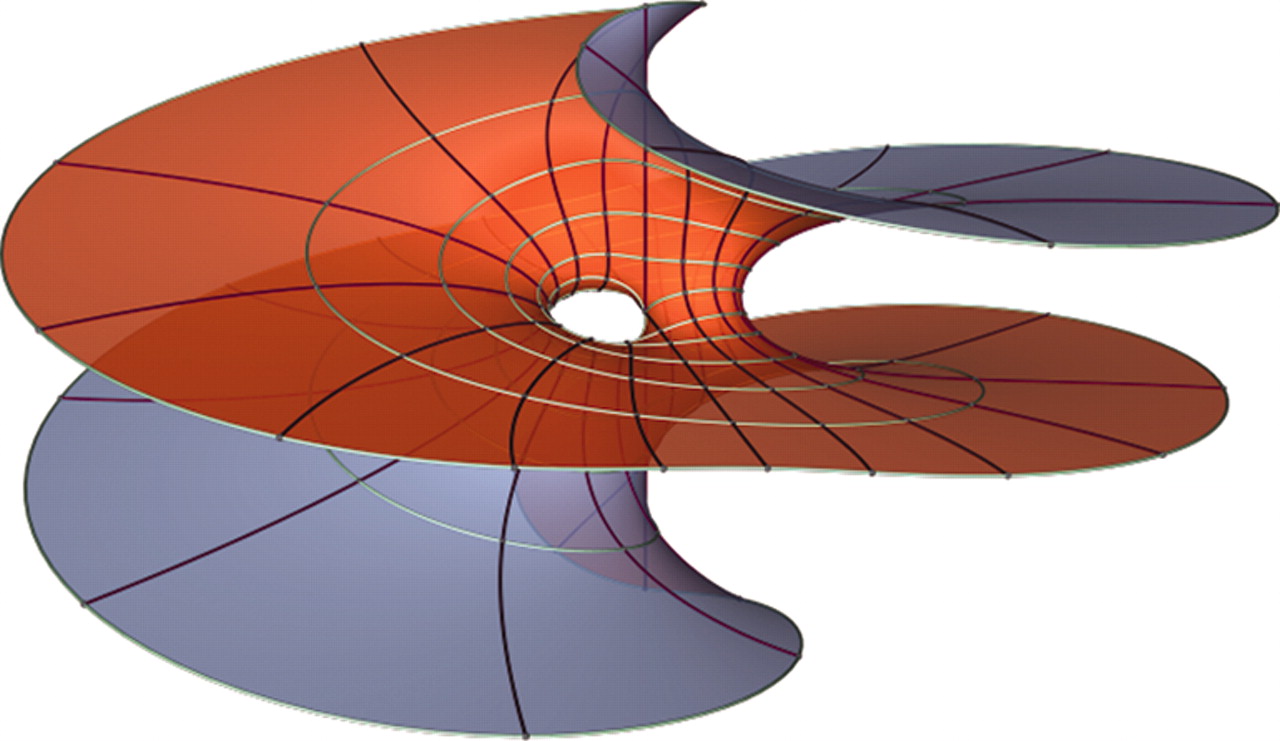}}   
\caption{A genus-$1$ helicoid.}\label{genus-one-figure}
\end{figure}

In general, it is very useful to have compactness theorems:
conditions on a sequence of minimal surfaces that guarantee
existence of a smoothly converging subsequence.

\begin{figure}
{\includegraphics[width=2in]{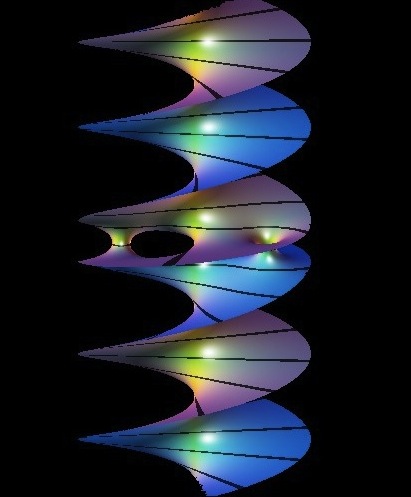}}
\caption{A genus-$2$ helicoid. (Picture by Martin Traizet (1993)).}\label{genus-two-figure}
\end{figure}

\begin{theorem}[basic compactness theorem]\label{basic-compactness-theorem}
Let $M_i$ be a sequence of minimal submanifolds of $\RR^n$ (or of
a Riemannian manifold) such that second fundamental forms are uniformly
bounded.  

Then locally there exist smoothly converging subsequences.
\end{theorem}

In particular, if $p_i\in M_i$ is a sequence 
bounded in $\RR^N$ and if $\dist(p_i, \partial M_i)\ge R>0$,
then (after passing to a subsequence), 
\[
    M_i \cap \BB(p_i, R)
\]
converges smoothly to a limit minimal surface $M^*$.
Here $\dist$ is {\bf intrinsic} distance in $M_i$, and $\BB(p_i, r)$ 
is the open geodesic ball of radius $r$ in $M_i$.

The theorem (if interpreted properly) is also true when $\dist$ is exterior
distance and $\BB(p,r)$ is the extrinsic open ball.
In this case the conclusion is that there is an $r$ with $0<r< R$ 
such that the connected component of $M_i\cap\BB(p_i,r)$
containing $p_i$ converges smoothly (after passing to a subsequence)
to a limit surface $M^*$.  Furthermore, under  rather mild hypotheses, the various $M^*$
will fit together to form a nice surface.  For example, if the $M_i$ are properly
immersed in an open set $\Omega$ and if the areas of the $M_i$ in compact
subsets of $\Omega$ are uniformly bounded, then a subsequence of the $M_i$
converges smoothly in $\Omega$ to a limit surface that is properly immersed
in $\Omega$.  
Even without area bounds, 
if the $M_i$ are properly embedded hypersurfaces in an open set $\Omega$,
then the $M^*$ fit together to form a 
lamination\footnote{A lamination is like a foliation except that gaps
are allowed.  For example, if $S\subset \RR$ is an arbitrary closed set, then the set of planes
$\RR^2\times \{z\}$ where $z\in S$ forms a lamination of $\RR^3$ by planes.}
 of $\Omega$ in which the leaves are smooth minimal hypersurfaces.

\begin{proof}[Proof of the Basic Compactness Theorem in $\RR^n$]
By scaling, we can assume that the principle curvatures are bounded by $1$,
and that $\dist(p_i, \partial M_i)\ge \pi/2$.

We can assume the $p_i$ converge to a limit $p$ and that
$\Tan(M_i,p_i)$ converge to a limit plane.  Indeed, in $\RR^n$,
we can assume by translating and rotating that $p_i\equiv 0$
and that $\Tan(M_i,p_i)$ is the horizontal plane through $0$.

For each $i$, let $S_i$ be the connected component of
\[
   M_i \cap (\BB^m(0,1/2) \times \RR^{N-m})
\]
containing $0$. 
The hypotheses imply that $S_i$ is the graph of
a function 
\[
    F_i: \BB^m(0,1/2) \to \RR^{N-m}
\]
where the $C^2$ norm of the $F_i$ is uniformly bounded.
Hence by Arzel\'a-Ascoli, we may assume (by passing to a subsequence)
that the $F_i$ converge in $C^{1,\alpha}$ to a limit function $F$.

So far we have not used minimality.
Since the surfaces $M_i$ are minimal, the $F_i$ are solutions to
an elliptic partial differential equation (or system of equations if $n>m+1$), 
the minimal surface equation (or system).  
According to the theory of such equations, convergence in $C^{1,\alpha}$
on $\BB^m(0,1/2)$ 
implies convergence in $C^{k,\alpha}$  on $\BB^m(0,1/2-\eps)$ (for any $k$ and $\eps$).

We have proved convergence in sufficiently small geodesic balls.  We leave it
to the reader to piece such balls together to get the convergence of the $M_i\cap\BB(p_i,R)$.
\end{proof}

This theorem indicates the importance of curvature estimates:
curvature estimates for a class of minimal surfaces
imply smooth subsequential convergence for sequences of such surfaces.

\section*{\quad The $4\pi$ curvature estimate}

\begin{theorem}[$4\pi$ curvature estimate~\cite{white-estimates}]\label{four-pi-estimate-theorem}
For every $\lambda< 4\pi$, there is a $C<\infty$ with the following property.
If $M\subset \RR^3$ is an orientable minimal surface with total curvature
$\le \lambda$, then 
\[
     |A(p)|\, \dist_M(p, \partial M)   \le  C.
\]
\end{theorem}

Here $|A(p)|$ is the norm of the second fundamental form of $M$ at $p$, i.e, the 
square root of the sum of the squares of the principal curvatures at $p$,
and $\dist_M$ denotes intrinsic distance in $M$.

The theorem is false for $\lambda=4\pi$, since the catenoid has total curvature $4\pi$
and is not flat.
(Earlier, Choi and Schoen \cite{choi-schoen} proved that there exists a $\lambda>0$ and 
a $C<\infty$ for which the conclusion holds.)

\begin{proof}
It suffices to prove the theorem when $M$ is a smooth, compact manifold with boundary,
since a general surface can be exhausted by such $M$.

Suppose the theorem is false.  Then there is a sequence $p_i\in M_i$ of examples
with  total curvature
$
   TC(M_i) \le \lambda
$
and with
\begin{equation}\label{blowing-up-equation}
    |A_i(p_i)| \, \dist(p_i, \partial M_i) \to \infty. 
\end{equation}
(Throughout the proof, all distances are intrinsic, so we write $\dist$ rather than $\dist_{M_i}$.)
We may assume that each $p_i$ has been chosen in $M_i$ to maximize
the left side of~\eqref{blowing-up-equation}.
By translating and scaling, we may also assume that $p_i=0$ and that $|A_i(p_i)|=1$,
and therefore that $\dist(0,\partial M_i)\to \infty$.
We may also replace $M_i$ by the geodesic ball of radius $R_i:=\dist(0,\partial M_i)$
about $0$.

Thus we have
\begin{gather*}
   |A_i(0)| = 1, \\
   R_i = \dist(0,\partial M_i) \to \infty, \, \text{and} \\
  |A_i(x)|\, \dist(x,\partial M_i) \le \dist(0,\partial M_i).
\end{gather*}
Now $\dist(0,x) + \dist(x,\partial M_i) = \dist(0,\partial M_i) = R_i$,
so  
\[
 |A_i(x)| \le \frac{R_i}{\dist(x,\partial M_i)} = \frac{R_i}{R_i- \dist(0,x)} 
 \le \frac{R_i}{R_i-r}
\]
if $\dist(x,\partial M_i)\le r$.

We have shown for each $r$ that
\[
   \sup_{\dist(x,0)\le r} |A_i(x)|  \le  \frac{R_i}{R_i-r}
   \to 1.
\]
Hence the $M_i$ converge smoothly
 by theorem~\ref{basic-compactness-theorem} (the basic compactness theorem) 
 to a complete
minimal surface $M$ with $|A_M(0)|=1$.

Thus by corollary~\ref{four-pi-plane-corollary} (the corollary to Osserman's Theorem), 
   $TC(M)\ge 4\pi$.
However, 
$  
   TC(M)\le \liminf_i TC(M_i) \le \lambda < 4\pi 
$,
a contradiction.
\end{proof}

\begin{remark*}
The theorem is also true (with the same proof) 
in $\RR^n$, but with $4\pi$ replaced by $2\pi$.
This is because Osserman's theorem is also true in $\RR^n$, but with $4\pi$ replaced by
$2\pi$.
\end{remark*}

Theorem~\ref{four-pi-estimate-theorem} can be generalized to manifolds
in various ways.  For example:

\begin{theorem}\label{four-pi-manifolds-theorem}
Suppose that $N$ is a $3$-dimensional submanifold of Euclidean space $\RR^n$ with the induced metric.
Let 
\[
  \rho(N) = (\sup |A_N|  + \sup |\nabla A_N|^{1/2})^{-1}.
\]
For every $\lambda<4\pi$, there is a $C=C_{\lambda,n}<\infty$ with the following
property.  If $M$ is an orientable, immersed minimal surface in $N$
and if the total absolute curvature of $M$ is at most $\lambda$,
then
\[
   |A_M(p)|\, \min\{ \dist_M(p, \partial M) , \rho(N)\} \le C
\] 
for all $p\in M$.
\end{theorem}

The proof is almost exactly the same as the proof of theorem~\ref{four-pi-estimate-theorem}.
In particular, we get a sequence $p_i\in M_i\subset N_i$ with $|A_{M_i}(p_i)|=1$
and with
\[
  \min \{ \dist_{M_i}(p_i,\partial M_i), \rho(N_i) \}  \to \infty.
\]
The fact that $\rho_{N_i}(p_i)\to \infty$ means that the $N_i$ are converging (in a suitable
sense) to $\RR^3$, so in the limit we get a complete  minimal immersed surface $M$
in $\RR^3$, exactly as in the proof of theorem~\ref{four-pi-estimate-theorem}.

\section*{\quad A general principle about curvature estimates}

Recall that we have proved:
\begin{enumerate}
\item\label{recall-global} A complete, nonflat minimal surface in $\RR^3$
has total curvature $\ge 4\pi$.
\item\label{recall-local} For any minimal $M\subset \RR^3$ with $TC(M)\le \lambda < 4\pi$, 
\[
  |A(p)|\dist_M(p, \partial M) \le C_\lambda.
\]
\end{enumerate}
We deduced~\eqref{recall-local} from~\eqref{recall-global}.  
But conversely,~\eqref{recall-local} implies~\eqref{recall-global}:
 if the $M$ in~\eqref{recall-global} is complete, then $\dist(p, \partial M)=\infty$,
so $|A(p)|=0$ according to~\eqref{recall-local}.
Thus~\eqref{recall-global} and~\eqref{recall-local} 
may be regarded as global and local versions of the same fact.

The equivalence of statements~\eqref{recall-global} and~\eqref{recall-local}  is an
 example of general principle:
any ``Bernstein-type" theorem (i.e., a theorem asserting that
certain complete minimal surfaces must be flat) should be equivalent
to a local curvature estimate.  Indeed, the Bernstein-type theorem in Euclidean
space should imply a local curvature estimate in arbitrary ambient manifolds
(as in theorem~\ref{four-pi-manifolds-theorem}).

\section*{\quad An easy version of Allard's Regularity Theorem}

As an example of the general principle discussed above, consider the following:

\begin{enumerate}
\item\label{allard-global}
Global theorem: If $M\subset \RR^n$ is a proper minimal submanifold without
boundary and if $\Theta(M)\le 1$, then $M$ is a plane.
\item\label{allard-local}
Local estimate: there exist $\lambda>1$, $\eps>0$, and $C< \infty$
with the following property.  If $M\subset \RR^N$ is minimal, $\dist(p,\partial M)\ge R$,
and $\Theta(M,p,R)\le \lambda$, then
\[
    \sup_{x\in \BB(p,\eps R)} |A(q)| \dist(q,\partial\BB(p,\eps R)) \le C,
\] 
where $\dist$ denotes Euclidean distance in $\RR^n$.  (Hence  
$|A(q)| \le 2C/(R\eps)$ for $q\in \BB(p, \eps R/2)$.)
\end{enumerate}

We have already proved statement~\eqref{allard-global} (see theorem~\ref{plane-density-one-theorem}).  
Statement~\eqref{allard-local}
 is a special case of  Allard's Regularity Theorem.

Clearly~\eqref{allard-local} implies~\eqref{allard-global}, 
and proof that~\eqref{allard-global} implies~\eqref{allard-local} is very similar to the 
proof of the $4\pi$ curvature estimate (theorem~\ref{four-pi-estimate-theorem}).   
Furthermore, as suggested
in the discussion of the general principle above, 
statement~\eqref{allard-global} implies a version of statement~\eqref{allard-local}
 in Riemannian manifolds.

Allard's theorem~(see \cite{allard}*{\S8} or \cite{simon-book}*{\S23--\S24})
 is much more powerful than statement~\eqref{allard-local}
because Allard does not assume 
that $M$ is smooth: it can be any minimal variety  
(``stationary integral varifold'').  He {\em concludes} that $M\cap \BB(p,\eps R)$
is smooth (with estimates).\footnote{I am describing Allard's theorem specialized
to minimal varieties.  His theorem is stated more generally for varieties with mean
curvature in $\mathcal{L}^p$ where $p$ can be any number
 larger than the dimension of the variety.
In this generality, the conclusion is not that $M\cap\BB(p,\eps)$ is smooth, but rather
that it is $C^{1,\alpha}$ for suitable $\alpha$.  If the variety is minimal, smoothness
then follows by standard PDE arguments. The easy version of Allard's Regularity Theorem
first appeared in~\cite{white-local-regularity}.}

\begin{quote}
{\bf Exercise}: Provide the details of the proof that (1) implies (2).
\end{quote}

\section*{\quad Bounded total curvatures}

For total curvatures that are bounded, but not bounded by some number $\lambda<4\pi$,
 we have the following theorem, which says that for a sequence of minimal surfaces
 with uniformly bounded total curvatures, we get smooth subsequential convergence
 except at a finite set of points where curvature concentrates:

\begin{theorem}[Concentration Theorem~\cite{white-estimates}]\label{concentration-theorem}
Suppose that $M_i\in \Omega\subset\RR^n$ are two-dimensional, orientable, minimal surfaces, 
that $\partial M_i\subset \partial \Omega$,
and that $TC(M_i)\le \Lambda <\infty$.

Then (after passing to a subsequence) there is a set $S\subset\Omega$ of 
at most $\frac{\Lambda}{2\pi}$
points such that $M_i$ converges smoothly in $\Omega\setminus S$
to a limit minimal surface $M$.  The surface $M\cup S$ is a smooth, 
possibly branched minimal
surface.

Now suppose that $\Omega\subset \RR^3$. Then $S$
contains at most $\frac{\Lambda}{4\pi}$ points.
Also,
if the $M_i$ are embedded, then $M\cup S$ is a smooth embedded surface
(with multiplicity, but without branch points.) 
\end{theorem}

The theorem remains true (with essentially the same proof) in  Riemannian
manifolds.

To illustrate the concentration theorem, let $M_k$ be obtained by dilating the catenoid by $1/k$.
Then $M_k$ converges to a plane with multiplicity $2$, and the 
convergence is smooth except at the origin.

Of course, the concentration theorem is only useful if the hypothesis (uniformly bounded
total curvatures of the $M_i$) holds in situations of interest.
Fortunately, there are many situations in which the hypothesis does hold.
For example, suppose the $M_i\subset \RR^n$ all have the same finite topological type.
 Suppose also that the boundary curves $\partial M_i$ are
 reasonably well-behaved:
 \[
  \sup_i  \int_{\partial M_i} |\kappa_{\partial M_i}|\,ds < \infty, 
 \]
 where $\kappa_{\partial M_i}$ denotes the curvature vector of the curve $\partial M_i$.
 (In other words, suppose that the boundary curves have uniformly bounded total
 curvatures.)
 Then the hypothesis $\sup_i TC(M_i) <\infty$ holds by the Gauss-Bonnet Theorem.

\begin{proof}[Proof of part of the concentration theorem in $\RR^3$]
Define measures $\mu_i$ on $\Omega$ by
\[
     \mu_i(U) = TC(M_i\cap U).
\]
By passing to a subsequence, we can assume that the $\mu_i$ converge
weakly to a limit measure $\mu$ with $\mu(\Omega)\le \Lambda$.

Let $S$ be the set of points $p$ such that $\mu\{p\} \ge 4\pi$.
 Then $|S|\le \frac{\Lambda}{4\pi}$, where $|S|$ is the number of points in $S$.
 
Suppose $x\in \Omega\setminus S$.  
Then $\mu\{x\} < \lambda <  4\pi$ for some $\lambda$.
 Thus there is a closed ball $\BB=\BB(x,r)\subset \Omega$ with $\mu(\BB)<\lambda$.
Hence
\[  
   TC(M_i\cap \BB) = \mu_i(\BB) < \lambda
\]
for all sufficiently large $i$.
Consequently, 
$|A_i(\cdot)|$ is uniformly bounded on $\BB(x, r/2)$ 
by the $4\pi$ curvature estimate (theorem~\ref{four-pi-estimate-theorem}).

Since $|A_i(\cdot)|$ is locally uniformly bounded in $\Omega\setminus S$,
 we get  subsequential smooth convergence on $\Omega\setminus S$
by the basic compactness theorem~\ref{basic-compactness-theorem}.

Let $p\in S$. By translation, we may assume that $p$ is the origin.
 Note that we can find $\BB(0,\eps)$ for which
$\mu(\BB(0,\eps)\setminus\{0\})$ is arbitrarily small.  
It follows (by the $4\pi$ curvature estimate~\ref{four-pi-estimate-theorem}
 and the basic compactness theorem~\ref{basic-compactness-theorem})
that if we dilate $M$ about $0$ by a sequence of numbers tending to infinity, 
 a subsequence of the dilated surfaces  converges smoothly
on $\RR^3\setminus\{0\}$ to a limit minimal surface
with total curvature $0$, i.e., to a union of planes.  By monotonicity, the number
of those planes is finite.   It follows that (for small $r$),  the surface
$M\cap (\BB(0,r)\setminus\{0\})$ is topologically a finite union of punctured disks.
In fact, it is not hard to show that the smooth subsequential convergence
of the dilated surfaces to planes implies that 
the components of $M\cap (\BB(0,r)\setminus \{0\})$ are not just topologically
punctured disks, but actually conformally punctured disks.

Let $F:D\setminus\{0\}\to \RR^3$ be a conformal (and therefore harmonic)
parametrization of one of those punctured disks.  Since isolated singularities
of bounded harmonic functions are removable, in fact $F$ extends smoothly
to $D$.  This proves that $M\cup S$ is a smooth (possibly branched) minimal
surface.

Finally, one can show that if $M\subset \RR^3$ is not an embedded surface (possibly with
multiplicity) then portions of it intersect each other transversely.  (This is true
for any minimal surface in a $3$-manifold.)  But
then the smooth convergence $M_i\to M$ away from $S$ implies that the $M_i$
would also have self-intersections.  In other words, if the $M_i$ are embedded, then $M\cup S$
is also embedded (possibly with multiplicity).
\end{proof}

If the surfaces $M_i$ in theorem~\ref{concentration-theorem}
 are simply connected, one can say more:

\begin{theorem}\label{simply-connected-theorem}
Suppose in the concentration theorem that 
$M_i\subset \Omega\subset \RR^3$ and that the $M_i$ are embedded and simply connected.
Then $S=\emptyset$.
\end{theorem}

\begin{proof}
Suppose to the contrary that $S$ contains a point $p$. 
Near $p$, the surface $M\cup\{p\}$ is a smooth embedded surface with some multiplicity $Q$. 
Thus we can  choose a small closed ball $\BB$ around $p$ 
such that $M\cap \partial \BB$ is very nearly
circular and so that $\BB\cap S=\{p\}$.  
The smooth convergence $M_i\to M$ away
from $S$ implies that (for large $i$) $M_i\cap \partial \BB$ is the union
of $Q$ very nearly circular curves.
(This is where we use embeddedness of the $M_i$: if the $M_i$ were not embedded, 
$M_i\cap \partial \BB$ might contain a component
that is perturbation of a circle transversed
multiple times.)
By the convex hull property (see exercise~\eqref{simply-connected-exercise} 
after theorem~\ref{convex-hull-theorem}), $M_i\cap \BB$ is a union of simply
connected components.   
Since (for large $i$) each such component has very nearly circular boundary, its total curvature
is close to $0$ by the Gauss-Bonnet Theorem.  But then the curvatures
of the $M_i$ are uniformly bounded on compact subsets of the interior of $\BB$
by the $4\pi$ curvature estimate (theorem~\ref{four-pi-estimate-theorem}).  
\end{proof}

Theorem~\ref{simply-connected-theorem} and its proof generalize to Riemannian $3$-manifolds, but 
one has to be careful because in some $3$-manifolds, simple connectivity
of a minimal surface $M$ does not imply that the components of its intersection with a small ball
(say a geodesic ball) are simply connected.  Thus one needs to make some
additional hypothesis.  For example, one could assume that the ambient space is simply connected
and has non-positive sectional curvatures or, more generally, that for each $p\in \Omega$
and $r>0$, the set $\{x\in \Omega: \dist(x,p)\le r\}$ has smooth boundary and that
the mean curvature vector of that boundary 
points into the set.   (See exercise~\ref{strictly-mean-convex-exercise} after
theorem~\ref{convex-hull-theorem}).

To apply the concentration theorem, we need uniform local bounds on total curvature.
Such bounds are implied by uniform local bounds on genus and area:

\begin{theorem}\cite{ilmanen-singularities}*{Theorem 3}\label{ilmanen-theorem}
Let $\Omega$ be an open subset of a smooth Riemannian manifold.
Suppose that $M_i$ are 
minimal surfaces in $\Omega$ with $\partial M_i\subset \partial \Omega$.
Suppose also that
\[
  \sup_i \operatorname{genus}(M_i)< \infty 
\]
and that
\[
  \sup_i \area(M_i\cap U)< \infty \qquad \text{for $U\subset\subset \Omega$}. 
\]
Then
\[
  \sup_i TC( M_i\cap U) < \infty  \qquad \text{for $U\subset\subset \Omega$}.
\]
\end{theorem}
Thus under the hypotheses of this theorem, 
we get the conclusion of the concentration theorem:
smooth convergence (after passing to a subsequence) away from a discrete set $S$.

Note: Ilmanen's Theorem~3 is about surfaces, not necessarily minimal, in Euclidean space;
it gives local bounds on total curvature (integral of the norm of the second fundamental form squared)
in terms of genus, area, and integral of the square of the mean curvature.
To deduce Theorem~\ref{ilmanen-theorem} from that result, isometrically embed $\Omega$
into a Euclidean space.

\section*{\quad Stability}

Let $M$ be a compact minimal submanifold of a Riemannian manifold.
We say that $M$ is {\bf stable} provided
\[
  \ddt^2_{t=0} \area(\phi_t M)\ge 0
\]
for all deformations $\phi_t$ with $\phi_0(x)\equiv x$ and $\phi_t(y)\equiv y$ 
for $y\in \partial M$.
For noncompact $M$, we say that $M$ is stable provided each compact
portion of $M$ is stable.

If $M\subset \RR^n$ is an oriented minimal hypersurface and if
$X(x)=\ddt_{t=0}\phi_t(x)$ is a normal vectorfield, we can write $X= u\nu$
where $u:M\to\RR$ and $\nu$ is the unit normal vectorfield.
Note that $\phi_t(x)=x$ for $x\in \partial M$ implies that $u\equiv 0$ on $\partial M$. 

\begin{theorem}[The second variation formula]
Under the hypotheses above, 
\begin{align*}
  \ddt^2_{t=0}\area(\phi_tM)
  &=
  \frac12\int_M( |\nabla u|^2 - |A|^2 u^2)\,dS  \\
  &=
  \frac12\int_M (-\Delta u - |A|^2) u \, dS.
\end{align*}
\end{theorem}

To prove the theorem, one observes that
 \[
      \ddt^2 \area(\phi_tM) = \int \ddt^2 J_m(D\phi_t)\, dS,
\]
 and calculates $\ddt^2 J_m(D\phi_t)$ as in the proof of the first variation
 formula.  (Integrate by parts to get the second expression from the first.)
 The formula remains true in an oriented ambient manifold $N$, except that one
 replaces $|A|^2$ by $|A|^2+\operatorname{Ricci}_N(\nu,\nu)$ where $\nu$
 is the unit normal vectorfield to $M$.
 See~\cite{simon-book}*{\S9} or~\cite{colding-minicozzi}*{1.\S8}, for example, for details.

The following theorem is one of the most important and useful facts about
stable surfaces.  It was discovered independently 
by Do Carmo  and Peng~\cite{do-carmo-peng}
and by Fischer-Colbrie and Schoen~\cite{fischer-colbrie-schoen}.
A few years later Pogorelov gave
another proof~\cite{pogorelov}.

\begin{theorem}\label{stability-theorem}
\begin{enumerate}
\item\label{global-stable} A complete, stable, orientable minimal surface in $\RR^3$ must be a plane.
\item\label{local-stable} If $M$ is a stable, orientable minimal surface in $\RR^3$, then
\[
   |A(p)|\, \dist(p, \partial M) \le C
\]
for some $C<\infty$.
\end{enumerate}
\end{theorem}

As usual, \eqref{global-stable} and~\eqref{local-stable} are equivalent.
Also, a version of~\eqref{local-stable} holds in Riemannian $3$-manifolds.
In some ways, Fischer-Colbrie and Schoen get the best results, because
they get results in $3$-manifolds of nonnegative scalar
curvature that include~\eqref{global-stable} as a special case.
However, the proof below is a slight 
modification\footnote{Here corollary~\ref{intrinsic-density-corollary} is used in place
 of one of the  lemmas that Pogorelov proves.} of
Pogorelov's.  First we prove some preliminary results.

\begin{theorem}[Fischer-Colbrie/Schoen]\label{eigenvalue-theorem}
Suppose $M$ is an oriented minimal hypersurface in $\RR^n$.
Then $M$ is stable if and only if there is a positive solution of
\[
   \Delta u + |A|^2 u = 0
\]
on $M\setminus \partial M$.
\end{theorem}

This is actually a very general fact about the lowest eigenvalue of self-adjoint, second-order
elliptic operators (first proved by Barta for the Laplace operator). 
In particular, theorem~\ref{eigenvalue-theorem}
is true in Riemannian manifolds with $|A|^2$ replaced by $|A|^2+\operatorname{Ricci}(\nu,\nu)$.
See~\cite{fischer-colbrie-schoen} or~\cite{colding-minicozzi}*{1.\S8, proposition 1.39} for the proof.

\begin{corollary}\label{universal-cover-corollary}
Let $M$  be as in theorem~\ref{eigenvalue-theorem}.
If $M$ is stable, then so is its universal cover.
\end{corollary}

\begin{proof} Lift the function $u$ from $M$ to its universal cover.
\end{proof}

\begin{proposition}\label{intrinsic-density-proposition}
Let $M$ be a complete, simply connected surface with $K\le 0$.
Let $A(r)=A_p(r)$ be the area of the geodesic ball $\BB_r$
 of radius $r$ about some point $p$.
Let
\[
  \theta(M) = \lim_{r\to \infty} \frac{A(r)}{\pi r^2}.
\]
Then
\[
   \theta(M) = 1 - \frac1{2\pi}\int_MK\,dS = 1 + \frac{TC(M)}{2\pi}.
\]
\end{proposition}

Note that $\theta(M)$ is an intrinsic analog of  $\Theta(M)$, the density
at infinity of a properly immersed minimal surface (without boundary) in Euclidean
space discussed in lecture~\ref{lecture1}.

\begin{proof} Let $L(r)$ be the length of $\partial B_r$.  
Then $A'=L$, 
so
\begin{align*}
A'' = L'   
= \int_{\partial B_r} k\,ds  
=
2\pi - \int_{B_r}K\,dS.
\end{align*}
(The formula for $L'$ is a special case of the first variation formula.)
Thus
\[
  \lim_{r\to\infty}A''(r) = 2\pi - \int_M K\,dS = 2\pi + TC(M).
\]
The result follows easily. 
\end{proof}

\begin{corollary}\label{intrinsic-density-corollary}
 If $M$ (as above) is a minimal surface in $\RR^3$ and if $\theta(M)<3$,
then $M$ is a plane.
\end{corollary}

\begin{proof} If $\theta(M)<3$, then $TC(M)<4\pi$ (by the proposition), and therefore
$M$ is a plane (by corollary~\ref{four-pi-plane-corollary}).
\end{proof}

\begin{lemma}[Pogorelov]\label{pogorelov-lemma}
Let $M\subset \RR^3$ be a simply connected, minimal immersed 
surface.   Suppose $B_R$ is a geodesic ball in $M$ of radius $R$
about some point $p\in M$ such that the interior of $B_R$ contains no points
of $\partial M$, i.e., such that $\dist(p,\partial M)\ge R$.

If $A(R):=\area(B_R) > \frac43 \pi R^2$, then $B_R$ is unstable.
\end{lemma}

\begin{proof} We may assume that $M=B_R$.
To prove instability, it suffices (by the second variation formula)
to find a function $u$ 
in $B_R$ with $u|\partial B_R=0$ such that $Q(u)<0$, where 
\begin{equation}\label{Q-equation}
  Q(u) = \int_M( |\nabla u|^2  -  |A|^2 u^2)\, dS = \int_M|\nabla u|^2\,dS + 2\int_M Ku^2\,dS
\end{equation}
(The second equality holds because $|H|^2=|A|^2 + 2K$ for any surface.) 
Let $r$ and $\theta$ be geodesic polar coordinates in $M$ centered
at the point $p$.  Thus the metric has the form
\[
   ds^2 = dr^2 + g^2\,d\theta^2
\]
for some nonnegative function $g(r,\theta)$ such that $g(0,0)=0$ and $g_r(0,0)=1$.
Recall that the Gauss curvature is given by
\[
   K = - \frac{g_{rr}}{g}.
\]
Thus the second integral in~\eqref{Q-equation} becomes
\begin{align*}
Q_2(u)
&=: 2\int_M u^2 K \,dS 
\\
&=
2\int_0^{2\pi}\int_0^R u^2 Kg \,dr\,d\theta
\\
&=
-2\int_0^{2\pi}\int_0^R u^2 g_{rr} \,dr\,d\theta.
\end{align*}
Integrating by parts twice gives
\begin{equation}\label{parts-done-equation}
\begin{aligned}
Q_2(u)
&=
4\pi u(0)^2 - 2\int_0^{2\pi}\int_0^R (u^2)_{rr} g\,dr\,d\theta
\\
&=
4\pi u(0)^2 - 4\int_0^{2\pi}\int_0^R (u_r)^2g\,dr\,d\theta - 4\int_0^{2\pi}\int_0^R uu_{rr}g\,dr\,d\theta
\\
&=
4\pi u(0)^2 - 4\int_M (u_r)^2 \,dS - 4\int_0^{2\pi}\int_0^R u u_{rr}g\,dr\,d\theta.
\end{aligned}
\end{equation}
Now let $u(r,\theta)=u(r)= (R-r)/R$, so that $u(r)$ decreases linearly from $u(0)=1$
to $u(R)=0$.  Then the last integral in~\eqref{parts-done-equation}
 vanishes, and $(u_r)^2=|\nabla u|^2=1/R^2$,
so combining~\eqref{Q-equation} and~\eqref{parts-done-equation} gives
\[
Q(u)
= 4\pi - \frac3{R^2} A(R),
\]
which is negative if $A(R)>\frac43\pi R^2$.
\end{proof}

Now we can give the proof of theorem~\ref{stability-theorem}:

\begin{proof}
By corollary~\ref{universal-cover-corollary},
 we may assume that $M$ is simply connected.
Suppose that $M$  is not a plane.
Then by corollary~\ref{intrinsic-density-corollary}, 
\[
    \theta(M)\ge 3 > \frac43,
\]
so $\frac{A(r)}{\pi r^2} > \frac43$
for large $r$.
But then  $M$ is unstable by lemma~\ref{pogorelov-lemma}.
\end{proof}


\lecture{Existence and Regularity of Least-Area Surfaces}\label{lecture4}

Our goal today is existence and regularity of a surface of least area
bounded by a smooth, simple closed curve $\Gamma$ in $\RR^N$.
As mentioned in lecture~\ref{lecture1}, the nature of such surfaces depends (in an interesting way!)
on what we mean by ``{\bf surface}",
``{\bf area}'', and ``{\bf bounded by}".   
There are different possible definitions of these terms, and they lead to
different versions of the Plateau problem.

In the most classical version of the Plateau problem:
``surface" means ``continuous 
    mapping $F:D\to \RR^n$ of a disk",
``bounded by $\Gamma$'' means ``such that $F:\partial D\to \Gamma$
  is a monotonic parametrization", and   
 ``area" means ``mapping area'' (as in multivariable calculus):
\[
    A(F) := \int J(DF)\,dS 
\]
where 
\[
    J(DF) 
    = \sqrt{ | F_x |^2 | F_y |^2 - ( F_x \cdot F_y )^2 }.
\]
($F_x = \partial F/ \partial y$ and $F_y = \partial F / \partial y$.)

\begin{theorem}[The Douglas-Rado Theorem]
Let $\Gamma$ be a smooth, simple closed curve in $\RR^N$.  
Let $\Cc$ be the class of continuous maps $F: \overline{D}\to \RR^N$
such that $F|D$ is locally Lipschitz and such that $F:\partial D\to \Gamma$
is a monotonic parametrization.
Then there exists a map $F\in \Cc$ that minimizes the mapping area $A(F)$.
 Indeed, there exists such a map that is harmonic and almost conformal,
and that is a smooth immersion on $D$ except (possibly) at isolated points 
(``branch points").
\end{theorem}

\begin{remark*}
The theorem remains true even if $\Gamma$ is just a continuous simple closed
curve, provided one assumes that the class $\Cc$ contains a map of finite area.
(If $\Gamma$ is smooth, or, more generally, if it has finite arclength, then $\Cc$ does
contain a finite-area map.)  With fairly minor modifications, the proof presented
below establishes the more general result.  See \cite{lawson}, for example, for
details.  
Even more generally, Douglas proved that $\Cc$ contains a harmonic, almost conformal map
without the assumption that $\Cc$ contains a finite-area map.
Morrey~\cite{morrey} generalized the Douglas-Rado theorem 
by replacing $\RR^n$ by
a general Riemannian $n$-manifold $N$ under a rather mild hypothesis (``homogeneous
regularity") on
the behavior of $N$ at infinity.  
\end{remark*}

We say that a continuous map $\phi:\partial D\to \Gamma$
is a monotonic parametrization  provided it is continuous, surjective, and has the following
property:  the inverse image of each point in $\Gamma$
is a connected subset of $\partial D$. Roughly speaking, this means that
if a point $p$ goes once around $\partial D$, always moving in one direction (e.g.,
counterclockwise), then $\phi(p)$ goes once around $\Gamma$, always in one direction.
(Note that $\phi$ is allowed to map arcs of $\partial D$ to a single points in $\Gamma$.)

Note that we need some condition on $F$ to guarantee that $A(F)$ makes sense.
Requiring that $F$ be locally Lipschitz on $D$ is such a condition, since such an $F$
is differentiable almost everywhere by Rademacher's Theorem.
(Alternatively, one could work in the Sobolev space of mappings whose first derivatives
are in $L^2$.)

The most natural approach for  proving existence for this (or any other minimization
problem) is the ``direct method", which we describe now.
Let $\alpha$ be the infimum of $A(F)$ among all $F\in \Cc$.
Then there exists a {\bf minimizing sequence} $F_i$, i.e., a sequence $F_i\in \Cc$
such that $A(F_i) \to \alpha$.  Now one hopes that there exists a 
subsequence $F_{i(j)}$ that converges to a limit $F\in \Cc$ with $A(F)=\alpha$.

For the direct method to work, one needs two ingredients: a compactness theorem (to guarantee
existence of a subsequential limit $F\in \Cc$), and lowersemicontinuity of the
functional $A(\cdot)$ (to guarantee that $A(F)=\alpha$.)

For the Plateau problem, a minimizing sequence need not have a
 convergent subsequence\footnote{In the geometric measure theory approach to Plateau's
 problem, one works with a class of surfaces and a suitable notion of convergence for which
 minimizing sequences do have convergent subsequences. One disadvantage (compared
 to the classical approach described here) is that a limit of simply connected surfaces need
 not not be simply connected.}.
For example, there exists a minimizing sequence $F_i$ such that the images
$F_i(D)$ converge as sets to all of $\RR^N$:
\begin{equation}\label{space-filling-equation}
    \text{$\dist(p, F_i(D)) \to 0$ for every $p\in \RR^N$}.
\end{equation}
(Think of $F_i(D)$ as a flat disk with a long, thin tentacle attached
near the center.  Even if the 
  tentacle is very long, its area can be made arbitrarily small by making it sufficiently
thin.  
By making the tentacle meander more and more as $i\to\infty$, 
we can arrange for~\eqref{space-filling-equation} to hold, even though $A(F_i)$ converges
 to the area of the flat disk.)

One can also find a minimizing sequence $F_i$ such that $F_i\,|\,D$ converges
pointwise to a constant map.  For example, suppose that $\Gamma$ is the unit
circle $x^2+y^2=1$ in the plane $z=0$, and let
\begin{align*}
   &F_k: \overline{D}\subset \RR^2  \to \RR^3\\
   &F_k(x,y)= (x^2+y^2)^k (x,y,0).
\end{align*}

To avoid such pathologies, instead of using an arbitrary minimizing sequence,
we choose a well-behaved minimizing sequence.
For that, we make use of the energy functional.
The {\bf energy} of a map $F: \Omega\subset \RR^2 \to \RR^N$  is
\[
   E(F) = \frac12 \int_M |DF|^2\,dS 
\]
where $|DF|^2 = |F_x|^2 + |F_y|^2$.

We need several facts about energy:

\begin{lemma}[Area-Energy Inequality]
For $F \in \Cc$, 
\[
    A(F) \le E(F),
\]
 with equality if and only if $F$ is almost conformal.
\end{lemma}

\begin{proof}
For any two vectors $\uu$ and $\vv$ in $\RR^n$, 
\[
    \sqrt{|\uu|^2 |\vv|^2 - (\uu \cdot \vv)^2} 
    \le
    \sqrt{|\uu|^2 |\vv|^2} 
    =
    |\uu|\,|\vv|
    \le 
    \frac12( |\uu|^2 + |\vv|^2),
\]
with equality if and only if $\uu$ are $\vv$ orthogonal and have the same length.
Apply that fact to $F_x$ and $F_y$ and integrate.
\end{proof}

\begin{lemma}\label{harmonic-lemma}
Suppose $F:\overline{D}\to \RR^N$ is smooth and harmonic.
Then
\[
    E(F) \le E(G)
\]
for all smooth $G:\overline{D}\to \RR^N$ with $G|\partial D=F|\partial D$, with
equality if and only if $G=F$.
\end{lemma}

\begin{proof} Let $V=G-F$.
Then
\begin{align*}
E(G)
&= E(F+V) 
 \\
&= E(F) + E(V) + \int DF\cdot DV\, dS  
\\
&= E(F)  + E(V) - \int \Delta F \cdot V\, dS  
\\
&= E(F) + E(V).  
\end{align*}  
\end{proof}
(The proof actually shows that  the lemma holds for domains in arbitrary Riemannian manifolds,
and in the Sobolev space of mappings whose first derivatives are in $L^2$.)

\begin{proof}[Proof of the Douglas-Rado Theorem]
Let $\alpha=\inf\{A(F): F\in \Cc\}$.
We begin with four claims, each of which implies that we can find
a minimizing sequence consisting of functions in $\Cc$ with some
additional desirable properties.

\begin{claim}
For every $\beta>\alpha$, there is a smooth map $F\in \Cc$ with $A(F)<\beta$.
\end{claim}

\begin{proof}[Proof of claim 1]
We will show that there is an $F\in \Cc$ such that $F$ is Lipschitz on $\overline{D}$,
such that $F$ is smooth near $\partial D$, and such that $A(F)< \beta$.
The assertion of claim 1 then readily follows by standard approximation theorems.

By definition of $\alpha$, there is an $G\in \Cc$ with $A(G)<\beta$.
Let $R>0$ be the {\bf reach} of the curve $\Gamma$, i.e., the supremum of numbers $\rho$
such that every point $p$ with $\dist(p,\Gamma)<\rho$ has a unique nearest point 
$\Pi(p)$ in $\Gamma$.   For $\delta<R/2$, let $\Phi_\delta:\RR^n\to \RR^n$ be the map
\[
\Phi_\delta(p)
=
\begin{cases}
p &\text{if $\dist(p,\Gamma)\ge 2\delta$}, \\
\Pi(p)  &\text{if $\dist(p,\Gamma)\le \delta$, and} \\
\Pi(p) + \left( \frac{\dist(p,\Gamma)}{\delta} - 1\right)(p-\Pi(p)) 
          &\text{if $\delta\le \dist(p,\Gamma) \le 2\delta$}.
\end{cases}
\]
Then for every $\delta\in (0,R)$, the map $\Phi_\delta\circ G$ is in the class $\Cc$.  Furthermore,
$A(\Phi_\delta\circ G)\to A(G)$ as $\delta\to 0$.   Now let $F=\Phi_\delta\circ G$
for a $\delta>0$ small enough that $A(F)<\beta$.

Note that there is an $r$ with $0<r<1$ such that $F$ maps the
annular region $A:=\{z: r\le |z|\le 1\}$ to $\Gamma$: $F(A)=\Gamma$. 
 Now it is straightforward to modify the definition of $F$ on $A$
so that $F(A)$ remains $\Gamma$, so that $F$ is Lipschitz, and so that
$F$ is smooth near $\partial D$ and maps $\partial D$ diffeomorphically to $\partial D$.
(Note that this modification does not change $A(F)$.)
\end{proof}

\begin{claim} If $\beta>\alpha$, then there exists a smooth map $G\in \Cc$
with $E(G)\le \beta$.
\end{claim}

Since $A(G)\le E(G)$ for every map $G$, claim 2 is stronger than claim 1.

\begin{proof}[Proof of claim 2]
By claim 1, there is a smooth map $F\in \Cc$ with $A(F)<\beta$.
Although $F$ is smooth, its image need not be a smooth surface.  That is, 
$F$ need not be an immersion.  
To get around this, for $\delta>0$, we define a new map
\begin{align*}
 &F_\delta: \overline{D} \to \RR^n\times \RR^2 \cong \RR^{n+2}, \\
 &F_\delta(z) = (F(z), \delta z).
\end{align*}
By choosing $\delta$ small, we can assume that $A(F_\delta)< \beta$.

Now $F_\delta(\overline{D})$ is a smooth, embedded disk.
Hence (by existence of conformal coordinates and the Riemann mapping
theorem), we can parametrize $F_\delta(\overline{D})$ by a smooth {\bf conformal}
map $\Phi: \overline{D}\to\RR^N$.  Let $G=\Pi\circ \Phi$, where 
  $\Pi:\RR^n\times\RR^2\to \RR^n$ is the projection map.  Then
\[
   E(G) \le E(\Phi) = A(\Phi) = A(F_\delta) < \beta,
\]
where $E(\Phi)=A(\Phi)$ by conformality of $\Phi$ and where $A(\Phi)=A(F_\delta)$
because $\Phi$ and $F_\delta$ parametrize the same surface.
\end{proof}

\begin{claim} For every $\beta>\alpha$, there is a smooth harmonic
map $F\in \Cc$
 such that $E(F)\le \beta$.
\end{claim}

\begin{proof}[Proof of claim 3]
By claim 2, there is a smooth map $G\in \Cc$ with $E(G)<\beta$.
Now let $F:\overline{D}\to\RR^n$ be the harmonic map
with the same boundary values as $G$.
By lemma~\ref{harmonic-lemma}, $E(F)\le E(G)<\beta$.
\end{proof}

\begin{claim}
Let $a$, $b$, and $c$ be three distinct points in $\partial D$, and let $\hat{a}$, $\hat{b}$, and $\hat{c}$
be three distinct points in $\Gamma$.
For every $\beta>\alpha$, there is a smooth harmonic map $F\in \Cc$
such that $E(F)<\beta$ and such that $F$ maps $a$, $b$, and $c$ to $\hat{a}$, $\hat{b}$, and $\hat{c}$.
\end{claim}

\begin{proof}
By claim 3, there is a smooth harmonic map $F\in \Cc$ such that $E(F)<\beta$.
Let $a'$, $b'$, and $c'$ be points in $\partial D$ that are mapped by $F$
to $\hat{a}$, $\hat{b}$, and $\hat{c}$.  Let $u:\overline{D}\to\overline{D}$
be the unique conformal diffeomorphism that maps $a$, $b$, and $c$
to $a'$, $b'$, and $c'$.  Then $F\circ u$ has the desired properties.
(For any map $F$ with a two-dimensional domain and for any conformal diffeomorphism $u$
of the domain, note that $E(F)=E(F\circ u)$, and that if $F$ is harmonic, then so is $F\circ u$.)
\end{proof}

By claim 4, we can find a sequence of smooth, harmonic maps $F_i\in \Cc$
such that
\[
  E(F_i) \to \alpha = \inf_{F\in \Cc} A(F).
\]
Furthermore, we can choose the $F_i$ so that they map $a$, $b$, and $c$ in $\partial D$
to $\hat{a}$, $\hat{b}$, and $\hat{c}$ in $\Gamma$.

By the maximum principle for harmonic functions (applied to  
$L\circ F_i$, for each linear function $L:\RR^n\to \RR$), the $F_i$ are uniformly bounded:
\begin{equation}\label{harmonic-convex-hull-equation}
   \max_{\overline{D}} |F_i(\cdot)| = \max_{\partial D}|F_i(\cdot)| = \max_{p\in \Gamma} |p|. 
\end{equation}
Thus by passing to a subsequence, we can assume that the $F_i$ converge
smoothly on the interior of the 
disk\footnote{For readers not familiar with this fact about harmonic maps
(which holds more generally for 
solutions of second-order linear elliptic partial differential equations
under mild conditions on the coefficients), 
note that each coordinate
of $F_i$ is the real part of a holomorphic function.  
By~\eqref{harmonic-convex-hull-equation}, those holomorphic functions take
values in a strip in the complex plane, and hence form a normal family.}
 to a harmonic map $F$.
However, we need uniform convergence on the closed disk.

\begin{claim*}[Equicontinuity] The maps $F_i$ are equicontinuous.
\end{claim*}

\begin{proof}[Proof of equicontinuity]
Suppose not.  Then (by the smooth convergence on the interior)
there exist point $p_i\in \partial D$ and $q_i\in \overline{D}$ such that
\[
   \delta_i:=|p_i-q_i|\to 0
\]
and such that $|F_i(p_i)-F_i(q_i)|\not\to 0$.
By passing to a subsquence (and by relabeling, if necessary) we may assume that the $p_i$
converge to a point $p\in \partial D$ that does not lie on the closed arc joining $b$ to $c$
(and disjoint from $a$.)   Let $E=\sup_i E(F_i)$.   By the Courant-Lebesgue Lemma  
(lemma~\ref{courant-lebesgue-lemma} below), there exist arcs
\[
   C_i = D\cap \partial \BB(p_i,r_i)
\]
with $r_i\in [\delta_i, \sqrt{\delta_i}]$ such that the arclength $L_i$ of $F|C_i$ satisfies
\[
  L_i \le \sqrt{ \frac{8\pi E}{|\ln(\delta_i)|} }
\]
which tends to $0$ as $i\to\infty$.

Let $D_i=D\cap \BB(p_i,r_i)$.  The boundary of $D_i$ consists of two arcs, $C_i$
and an arc $C_i'$ in $\partial D$, namely $\BB(p_i,r_i)\cap \partial D$.  The two arcs have the 
same endpoints.  Since the length of $F(C_i)$ tends to $0$, the distance between the endpoints
tends to $0$.

Thus $F(C_i')$ is an arc in $\Gamma$, and the distance between the endpoints tends to $0$.
Thus, for large $i$, $F(C_i)$ is either a (i) very short arc in $\Gamma$ or (ii) all of $\Gamma$ except
for a very short arc.
Since $F(C_i')$ contains $F(p)$ and (for large $i$) is disjoint from the arc in $\Gamma$
joining $\hat{b}$ to $\hat{c}$, in fact $F(C_i)$ must be very short arc in $\Gamma$: its length
tends to $0$ as $i\to\infty$.

We have shown that the arclength and therefore the diameter\footnote{The diameter
of a subset of a metric space is the supremum of the distance between pairs of points in the subset.}
of $F(\partial D_i)$ tends to $0$.  By the maximum principle for harmonic functions,
$F_i(\overline{D_i})$ is contains in the convex hull of $F_i(\partial D_i)$, so the diameter
of $F_i(\overline{D_i})$ tends to $0$. Therefore $|F(p_i)-F(q_i)|\to 0$.
This completes the proof of equicontinuity.
\end{proof}

By equicontinuity, we can (by passing to a subsequence) assume
that the $F_i$ converge uniformly on $\overline{D}$ to a limit map $F$.  
As already mentioned, $F$ is harmonic on the interior.  The uniform convergence
implies that $F\in \Cc$, so 
\[
 \alpha \le A(F) \le E(F) \le \liminf E(F_i) \le \alpha. 
\]
Since $A(F)=E(F)$, the map is almost conformal.
\end{proof}

\begin{lemma}[Courant-Lebesgue Lemma]\label{courant-lebesgue-lemma}
Let $\Omega\subset \RR^2$ and $F:\Omega \to \RR^n$ be a map with energy $E$.
Let $p$ be a point in $\RR^2$ and let $L(r)$ be the arclength of $F|\partial \BB(p,r)$.
Then
\[
   \int_{0}^\infty \frac{L(r)^2}r\,dr \le 4\pi E.
\]
Consequently, 
\[
   \min_{a\le r \le b} L(r)^2 \le  \frac{4\pi E}{\ln(b/a)}.
\]
\end{lemma}

\begin{proof}
It suffices to consider the case $p=0$.  Using polar coordinates,
\[
    |DF|^2 =   |F_r|^2 + \frac1{r^2}|F_\theta|^2 \ge \frac1{r^2}|F_\theta|^2.
\]
Thus
\begin{align*}
   L(r)^2 
   &= \left( \int_0^{2\pi} F_\theta\,d\theta\right)^2
   \\ 
   &\le 2\pi \int_0^{2\pi} | F_\theta |^2\,d\theta.
   \\
   &\le 2\pi r^2\int_0^{2\pi} |DF|^2\,d\theta.
\end{align*}
Therefore
\begin{equation*}
\int \frac{L(r)^2}{r}\,dr 
\le  2\pi \int_{r=0}^\infty \int_{\theta=0}^{2\pi} |DF|^2\, r \,d\theta \,dr  \\
=4\pi E.
\end{equation*}
\end{proof}

\section*{\quad Boundary regularity}

The Douglas-Rado Theorem produces an almost conformal, harmonic
map $F$ that is continuous on the closed disk and is such that $F|\partial D$ gives
a monotonic parametrization of the curve $\Gamma$.  
 It is not hard to show that any such map (whether or not it minimizes area) cannot be
constant on any arc of $\partial D$. 
(See for example \cite{osserman-book}*{lemma~7.4} or \cite{lawson}*{proposition~11}.)
  It follows from the monotonicity of $F|\partial D$
that $F:\partial D\to \Gamma$ is a homeomorphism.
Later, every such map was proved to be smooth on the closed disk provided $\Gamma$
is smooth.
Roughly speaking, such a map $F:\overline{D}\to \RR^n$ turns out to be
as regular as $\Gamma$.  For example, if $\Gamma$ is $C^{k,\alpha}$ for some $k\ge 1$
and $\alpha\in (0,1)$, then so is $F$, and if $\Gamma$ is analytic, then so is $F$.
(Lewy first proved that minimal surfaces in $\RR^n$ with analytic boundary curves are
analytic up to the boundary.  The fundamental breakthrough was due to Hildebrandt~\cite{hildebrandt},
who, in the case of area-minimizing surfaces, extended Lewy's result to arbitrary ambient manifolds
and who also proved the corresponding result for $C^4$ boundaries.  
Later Heinz and Hildebrandt~\cite{HH} proved such results for surfaces that are minimal but not
necessarily area minimizing.  See also~\cite{kinderlehrer}.)

\section*{\quad Branch points}

Let $F:D\to \RR^N$ be a non-constant, harmonic, almost conformal
map (such as given by the Douglas-Rado theorem).

Recall that harmonicity of $F$ means that the map 
$F_z = \frac12(F_x - i F_y)$ from $D$ to $\CC^n$ is holomorphic.
Thus $F_z$ can vanish only at isolated points. 
Those points are called ``branch points".  Away from the branch points,
 the map is a smooth, conformal immersion.

Using the Weierstrass representation, it is easy to give
examples of minimal surfaces with branch points.
(The branch points are the points where $g$ has a pole of order $m$ (possibly $0$)
and where $\nu$ has a zero of order strictly greater than $2m$.)
But are there area-minimizing examples?
The following theorem implies that there are such examples in $\RR^n$ for $n\ge 4$:

\begin{theorem}[Federer, following Wirtinger]\label{federer-wirtinger-theorem}
Let $M$ be a complex variety in $\CC^n$.  Then (as a real variety in $\RR^{2n}$)
$M$ is absolutely area minimizing in the following sense: if $S$ is a compact
portion of $M$, and if $S'$ is an oriented variety with the same oriented boundary
as $S$, then $\area(S)\le \area(S')$.
\end{theorem}

Here ``with the same oriented boundary" means that $\partial S'=\partial S$ and
that $S'$ and $S$ induce the same orientation on the boundary.  For the proof,
see \cite{federer-wirtinger} or \cite{lawson}*{pp.~37--40}.

Using the Federer-Wirtinger Theorem, we can give many examples
of branched, area-minimizing surfaces.
For example, the map
\begin{align*}
&F: D\subset \RR^2 \cong \CC^2 \to \RR^4\cong \CC^2 \\
&F(z)= (z^2, z^3)
\end{align*}
has a branch point at the origin and 
 is area-minimizing by the Federer-Wirtinger Theorem.
 
Whether there exist any examples other than the ones provided by the Federer-Wirtinger
Theorem is a very interesting open question.  In other words, must a connected 
least-area 
surface with a true\footnote{A branch point $p\in D$ of $F$ is called {\bf false}
if there is a neighborhood $U$ of $p$ such that the image $F(U)$ is a smooth, embedded
surface. Otherwise the branch point is {\bf true}.  For example,
 if $F: D\subset \CC \to \RR^n$ is a smooth
immersion, the $z\mapsto F(z^2)$ has a false branch point at $z=0$.}
branch point in $\RR^{2n}\cong \CC^n$ be holomorphic after a suitable rotation
of $\RR^{2n}$?   
The paper~\cite{micallef-white} is suggestive in this regard.

\section*{\quad The theorems of Gulliver and Osserman}

Osserman and Gulliver proved in $\RR^3$ (or more
generally in any Riemannian $3$-manifold) that the Douglas-Rado solution
cannot have any {\bf interior} branch 
points.\footnote{Osserman ruled out true branch points in $\RR^3$, and
Gulliver extended Osserman's result to $3$-manifolds and also ruled
out false branch points.  
Alt~\cite{alt} independently proved some of Gulliver's results.
See \cite{colding-minicozzi}, \cite{lawson},
 \cite{dierkes-et-al}, or the original papers for details.}  
Thus (away from the boundary), the map $F$ is a smooth immersion.

Whether the map $F$ in the Douglas-Rado Theorem can have boundary branch points
(for a $3$-dimensional manifold) is one of longest open questions in minimal surface theory.
Using the Federer-Wirtinger Theorem, one can give examples in $\RR^n$ for $n\ge 4$,
such as
\begin{align*}
&F: \{x+iy: x \ge 0\} \to \CC^2 \cong \RR^4 \\
&F(z) = (z^3, e^{-1/\sqrt{z}}).
\end{align*}

There are some situations in which boundary branch points
are known not to occur:
\begin{enumerate}
\item\label{easy-boundary-point}
 If $\Gamma$ lies on the boundary of a compact, strictly convex region in $\RR^n$.
  In this case, one need not assume area minimizing: minimality suffices.
  (The proof is a slight modification of the proof of 
  theorem~\ref{convex-hull-theorem}, together with the Hopf
  boundary point theorem.)
\item If $\Gamma$ is a real analytic curve in $\RR^n$ 
     or more generally in an analytic Riemannian manifold~\cite{white-boundary-branch}.
\end{enumerate}

\section*{\quad Higher genus surfaces}

Let $\Gamma$ be a simple closed curve in $\RR^n$.
Does $\Gamma$ bound a least-area surface of genus one?
Not necessarily. Consider a planar circle $\Gamma$ in $\RR^3$.
By the convex hull principle (theorem~\ref{convex-hull-theorem}), $\Gamma$ bounds only one minimal surface:
the flat disk $M$ bounded by $\Gamma$.
We can take a minimizing sequence of genus one surfaces, but (for this
example) in the limit, the handle shrinks to point, and we end up with the disk.

Technically speaking, the planar circle does bound a least area
genus $1$ suface in the sense of mappings.
Let $\Sigma$ be a smooth genus $1$ surface consisting a a disk with 
a handle attached.  There is a smooth map $F: \Sigma\to M$
that collapses the handle to the center $p$ of the disk $M$ bounded by $\Gamma$,
 and that maps the rest of $\Sigma$ 
diffeomorphically to $M\setminus \{p\}$.
However, there is no ``nice'' area-minimizing map $F: \Sigma\to \RR^3$ with
boundary $\Gamma$. For example, there is no such map that is an immersion
except at isolated points.

\begin{definition*}
Let $\Gamma$ be a smooth, simple closed curve in $\RR^n$.
Let $\alpha(g)$ be the infimum of the area of genus $g$ surfaces
bounded by $\Gamma$.
\end{definition*}

\begin{proposition}\label{alpha-decreasing-proposition}
$\alpha(g)\le \alpha(g-1)$.
\end{proposition}

\begin{proof} Take a surface of genus $g-1$ whose area is close to $\alpha(g-1)$,
and then attach a very small handle. 
\end{proof}

\begin{theorem}[Douglas\footnote{It seems that Douglas never gave a complete proof of
``Douglas's Theorem".
A result very similar to Douglas's Theorem, but for minimal surfaces without boundary
in Riemannian manifolds, was proved by Schoen and Yau~\cite{schoen-yau}.
Later, Jost~\cite{jost-douglas} gave a complete proof of Douglas's original theorem.}]\label{douglas-theorem}
If $\alpha(g)<\alpha(g-1)$, then there exists a domain $\Sigma$ consisting
of a genus $g$ Riemanan surface with an open disk removed, and a continuous map
\[
  F: \Sigma \to \RR^n
\]
that is harmonic and almost conformal in the interior of $F$, that maps $\partial \Sigma$
monotonically onto $\Gamma$, and  that has area $A(F)$ equal to $\alpha(g)$.
\end{theorem}

The proof is similar to the proof of the Douglas-Rado Theorem, but more complicated
because not all genus $g$ domains are conformally equivalent. 
For example, up to conformal equivalence, there is a $3$-parameter
family of genus-one domains with one boundary component.  As a result, we have
to vary the domain as well as the map.

The Douglas theorem can be restated slightly informally as follows:

\begin{theorem}
Let $g$ be a nonnegative integer.   The least area among
all surfaces of genus $\le g$ bounded by $\Gamma$ is attained
by a harmonic, almost conformal map.
\end{theorem}

\begin{proof}[Proof (using the Douglas Theorem)]
Let $k$ be the smallest integer such that $\alpha(k)=\alpha(g)$. 
Then $0\le k\le g$. 
If $k=0$, then the Douglas-Rado solution is a disk that attains the desired
infimum $\alpha(g)=\alpha(0)$.
If $k>0$, then $\alpha(k)< \alpha(k-1)$, so the genus $k$ surface given
by the Douglas Theorem attains the desired infimum $\alpha(g)=\alpha(k)$.
\end{proof}

The theorems of Gulliver and Osserman also hold for these higher genus surfaces:
in $\RR^3$ (and in $3$-manifolds) 
they must be smooth immersions except possibly at the boundary.

{\bf Summary}: For a fixed genus $g$ and curve $\Gamma$, we
cannot in general minimize area among surfaces of genus equal to $g$ 
and get a nice surface: minimizing sequences may converge to
surfaces of lower genus.
However, we can always minimize area among surfaces of genus $\le g$:
the minimum will be attained by a harmonic, almost conformal map.
Intuitively, the Douglas theorem is true because when we take the limit
of a minimizing sequence of genus $g$ surfaces, we can loose handles
but we cannot gain them.

\section*{\quad What happens as the genus increases?}

Fix a smooth, simple closed curve $\Gamma$ in $\RR^3$.
As above, we let $\alpha(g)$ denote the least area among genus $g$ surfaces
bounded by $\Gamma$.  
According to proposition~\ref{alpha-decreasing-proposition}, $\alpha(g)$ is a decreasing
function of $g$.  The following provides a sufficient condition for $\alpha(g)$
to be strictly less than $\alpha(g-1)$. (Recall that by the Douglas Theorem, the
strict inequality implies existence of a least-area genus $g$ surface bounded
by $\Gamma$.)

\begin{theorem}\label{surgery-theorem}
Suppose $M\subset \RR^3$ is a minimal surface of genus $(g-1)$ bounded by $\Gamma$.
Suppose also that $M\setminus \Gamma$ is not embedded.  
Then $\Gamma$ bounds a genus $g$ surface whose area is strictly less than the
area of $M$.  In particular, if $\area(M)=\alpha(g-1)$, then $\alpha(g)<\alpha(g-1)$.
\end{theorem}

\begin{proof}
One can show that if $M\setminus \Gamma$ is not embedded, then there is a curve $C$
along which two portions of $M$ cross transversely. (There may be many such curves.)
We will use that fact without proof here.  Note that we can cut and paste $M$ along an arc of $C$
to get a new surface $M^*$.   There are two ways to do the surgery: one produces an
orientable surface and the other a non-orientable surface.  We do the surgery that makes
$M^*$ orientable.  The new surface is piecewise smooth but not smooth.
It has the same area as $M$ and has genus $g$.
By rounding the corners of $M^*$, we can make a new genus $g$ surface
whose area is strictly less than $\area(M^*)=\area(M)$.
\end{proof}

\begin{theorem}\label{knot-genus-theorem}
For each $g$, there exists a smooth, simple closed curve $\Gamma$ in $\RR^3$
such that $\alpha(0)>\alpha(1) > \dots > \alpha(g)$, and such that
for every $k< g$, each genus-$k$ least area surface is non-embedded.
\end{theorem}

\begin{proof}
The {\bf genus} of a simple closed curve in $\RR^3$ is defined to be the smallest
genus of any embedded minimal surface bounded by the curve.
Using elementary knot theory, one can show that there are smooth curves
of every genus.   Let $\Gamma$ be such a curve of genus $g$.
For $k=0, 2, \dots, g-1$, let $M_k$ be a least-area surface of genus $\le k$
bounded by $\Gamma$, so that $\area(M_k)=\alpha(k)$.
Since $\Gamma$ has genus $g>k$, the surface $M_k$ cannot be embedded.
Therefore $\alpha(k+1)<\alpha(k)$ by theorem~\ref{surgery-theorem}.
\end{proof}

Actually, the relevant notion is not the genus of $\Gamma$, but rather the
``convex hull genus'' of $\Gamma$: the smallest possible genus of an embedded
surface bounded by $\Gamma$ and lying in the convex hull of $\Gamma$.

\begin{theorem}[Almgren-Thurston \cite{almgren-thurston}]\label{almgren-thurston-theorem}
For every $\eps>0$ and for every positive integer $g$, 
there exists a smooth, unknotted, simple closed
curve $\Gamma$ in $\RR^3$ whose convex hull genus is $g$
and whose total curvature is less than $4\pi +\eps$.
\end{theorem}

(Recall that the total curvature of a smooth curve is the integral with respect
to arclength of the norm of the curvature vector.)

Later Hubbard~\cite{hubbard} gave a beautiful, very simple proof of this theorem
and gave an explicit formula for calculating the convex hull genus of
a large, interesting family of curves.

\begin{theorem}\label{unknotted-alpha-theorem}
For every $\eps>0$ and for every positive integer $g$, there exists
a smooth, unknotted simple closed curve $\Gamma$ in $\RR^3$
with total curvature $\le 4\pi+\eps$
such that $\alpha(0)>\alpha(1)>\dots >\alpha(g)$, and such that
for every $k< g$, each genus-$k$ least area surface is non-embedded.
\end{theorem}

\begin{proof}
Let $\Gamma$ be a curve satisfying the conclusion of
 theorem~\ref{almgren-thurston-theorem}.
By the convex hull property (theorem~\ref{convex-hull-theorem}),
 any embedded minimal surface bounded
by $\Gamma$ has genus $\ge g$.  The rest of the proof is exactly
the same as the proof of theorem~\ref{knot-genus-theorem}.
\end{proof}

However, for a smooth curve, eventually the function $\alpha(\cdot)$ must stabilize
according to the following theorem of Hardt and Simon~\cite{HS}:

\begin{theorem}
Let $\Gamma$ be a smooth simple closed curve in $\RR^3$.
Let $\alpha=\inf \alpha(\cdot)$ be the infimum of the areas of all
orientable surfaces bounded by $\Gamma$.  Then
\begin{enumerate}
\item The infimum is attained, and any surface that attains
 the infimum is smoothly embedded (including at the boundary).
\item  The set of surfaces that attain the infimum is finite.
\end{enumerate}
\end{theorem}

In particular, if $g$ is the genus of a surface that attains the infimum,
then the $\alpha(g)\equiv \alpha(k)$ for all $k\ge g$.

On the other hand, one can construct a simple closed curve $\Gamma$
that is smooth except at one point such that $\alpha(g)>\alpha(g+1)$
for all $g$.   For example, take such a curve of infinite genus or
or even just of infinite convex hull genus, or
see \cite{almgren-expo} for an example (due to Fleming~\cite{fleming-example})
for which $\alpha(\cdot)$ is strictly decreasing and for which the Douglas
solutions are all embedded.
Indeed, all kinds of pathologies can happen once one allows a point
at which the curve is not smooth:

\begin{theorem}\cite{white-bridge2}*{1.3}
There exists a simple closed curve $\Gamma$ in $\partial \BB(0,1)\subset \RR^3$
and a number $A<\infty$ such that $\Gamma$ is smooth except at one
point $p$ and such that the following holds:
for every area $a\in [A,\infty]$, for every genus $g$ with $0\le g \le \infty$, and
for every index $I$ with $0\le I \le \infty$, the curve $\Gamma$ bounds uncountably
many embedded minimal surfaces that are smooth except at $p$
and that have area $a$, genus $g$, and 
index\,\footnote{See, for example,~\cite{colding-minicozzi}*{1.8} for the definition
and the basic properties of the index.} of instability $I$.
\end{theorem}

This is in sharp contrast to the case of an everywhere smooth, simple closed
curve $\Gamma$ in the boundary of a convex set in $\RR^3$.
For such a curve, one can show that for each genus $g<\infty$, the set of 
embedded genus-$g$ minimal surfaces bounded by $\Gamma$ is compact
with respect to smooth convergence~\cite{white-estimates}.
It follows that (for each $g$)
the set of possibly indices of instability is finite.
With a little more work, one can 
show that the set of areas of such surfaces (for each $g$)
is a finite set.
Of course, if $\Gamma$ is smooth, then 
the areas of all the minimal surfaces (regardless of genus)
are bounded above according to theorem~\ref{area-formula-theorem}.

\section*{\quad Embeddedness: The Meeks-Yau Theorem}

\begin{theorem}[Meeks-Yau~\cite{meeks-yau}]
Let $N$ be a Riemannian $3$-manifold and 
let $F: \overline{D}\to N$ be a least-area disk (parametrized almost conformally)
with a smooth boundary curve $\Gamma$.  
Suppose $F(D)$ is disjoint from $\Gamma$.
Then $F$ is a smooth embedding.
\end{theorem}

The disjointness hypothesis holds in many situations of interest.
In particular, it holds if $\Gamma$ lies on the boundary
of a compact, convex subset of $\RR^3$. (This follows from the
strong maximum principle.)
More generally, it holds if $N$ is a 
a compact, mean convex $3$-manifold  and if $\Gamma$ lies in $\partial N$.
(Mean convexity of $N$ means that the mean curvature vector at each point
of the boundary is a nonnegative multiple of the inward unit normal.)

\begin{proof}[Idea of the proof]
Suppose $M$ is immersed but not embedded.
One can show that it contains an arc along which it intersects itself transversely.
One can cut and paste $M$ along such arcs 
to get a new piecewise smooth (but not smooth) surface $\tM$.
Such surgery is likely to produce a surface of higher genus
 (as in the proof of theorem~\ref{surgery-theorem}).
However, Meeks and Yau show that it is possible to do the surgery 
(simultaneously on many arcs)
in such a way that $\tM$ is still a disk.  Thus
\[
    \area(\tM) = \area(M),
\]
so $\tM$ is also area minimizing. 
However, where $\tM$ has corners, one can round the corners to 
get a disk with less area than $\tM$, a contradiction.
\end{proof}

\newcommand{\hide}[1]{}









\end{document}